\documentclass{article}

\usepackage{color}             
\usepackage[all,color]{xy}     
\usepackage{graphicx}          
\usepackage{amsmath}           
\usepackage{amsfonts}          
\usepackage{amssymb}           
\usepackage{mathrsfs}          
\usepackage{amsthm}            
\usepackage{enumerate}         
\usepackage{manfnt}            
\usepackage{marvosym}          
\usepackage{pifont}            
\usepackage[ansinew]{inputenc} 
\usepackage{float}             
\usepackage{booktabs}
\usepackage{verbatim}          
\usepackage{subfigure}         
\usepackage{hyperref}          
\usepackage{cancel}            

\usepackage{mathtools}
\usepackage{units}
\usepackage{t1enc}
\usepackage{xspace}
\usepackage{epsfig}  
\usepackage{xspace}            
\usepackage{oldgerm} 
\usepackage{url}               
\usepackage{latexsym}
\usepackage{supertabular}
\newtheorem{thm}{Theorem}
\newtheorem{cor}[thm]{Corollary}
\newtheorem{lem}[thm]{Lemma}

\theoremstyle{remark}

\theoremstyle{definition}
\newtheorem{dfn}[thm]{Definition}

\newtheorem{rmk}[thm]{Remark}
\newtheorem{notation}[thm]{Notation}

\newcommand{\ndivides}{\ensuremath{\nmid}}

\newcommand{\bigdotcup}{\ensuremath{\dot{\bigcup}}}

\newcommand{\N}{\ensuremath{\mathbb{N}}}
\newcommand{\Z}{\ensuremath{\mathbb{Z}}}
\newcommand{\R}{\ensuremath{\mathbb{R}}}
\newcommand{\C}{\ensuremath{\mathbb{C}}}
\newcommand{\Q}{\ensuremath{\mathbb{Q}}}
\renewcommand{\H}{\ensuremath{\mathbb{H}}}

\newcommand{\Zp}{\ensuremath{\mathbf{Z}_p}\xspace}

\newcommand{\Ztwo}{\ensuremath{\mathbf{Z}_2}\xspace}

\newcommand{\A}{\mathcal{A}}

\newcommand{\e}{\mathfrak{e}}

\renewcommand{\a}{\textfrak{a}}
\renewcommand{\b}{\textfrak{b}}
\renewcommand{\c}{\textfrak{c}}

\newcommand{\im}{\ensuremath{\operatorname{image}}}
\renewcommand{\Im}{\ensuremath{\operatorname{Im}}}

\newcommand{\chr}{\ensuremath{\operatorname{char}}}
\newcommand{\sign}{\ensuremath{\text{sign}}}
\newcommand{\SL}{\ensuremath{\operatorname{SL}}}
\newcommand{\GL}{\ensuremath{\operatorname{GL}}}

\renewcommand{\ker}{\ensuremath{\text{ker}}}
\newcommand{\id}{\ensuremath{\text{id}}}

\newcommand{\smallsum}{\ensuremath{\textstyle \sum}}

\newcommand{\mat}[1]{\ensuremath{\begin{pmatrix} #1 \end{pmatrix}}}
\newcommand{\tmat}[1]{\ensuremath{\left(\begin{smallmatrix} #1 \end{smallmatrix}\right)}}

\newcommand{\one}{\ensuremath{\mathbf 1}}

\newcommand{\ideal}[1]{\ensuremath{\left\langle{#1}\right\rangle}}
\newcommand{\bilform}[1]{\ensuremath{\langle{#1}\rangle}}

\DeclareFontFamily{U}{matha}{\hyphenchar\font45}
\DeclareFontShape{U}{matha}{m}{n}{
      <5> <6> <7> <8> <9> <10> gen * matha
      <10.95> matha10 <12> <14.4> <17.28> <20.74> <24.88> matha12
      }{}
\DeclareSymbolFont{matha}{U}{matha}{m}{n}
\DeclareFontFamily{U}{mathx}{\hyphenchar\font45}
\DeclareFontShape{U}{mathx}{m}{n}{
      <5> <6> <7> <8> <9> <10>
      <10.95> <12> <14.4> <17.28> <20.74> <24.88>
      mathx10
      }{}
\DeclareSymbolFont{mathx}{U}{mathx}{m}{n}

\DeclareMathSymbol{\obot}         {2}{matha}{"6B}
\DeclareMathSymbol{\bigobot}       {1}{mathx}{"CB}

\newcommand{\orthplus}{\ensuremath{\obot}}

\newcommand{\D}{\mathcal{D}}

\renewcommand{\a}{\textfrak{a}}
\renewcommand{\b}{\textfrak{b}}
\newcommand{\F}{\ensuremath{\mathcal{F}}}
\newcommand{\FF}{\ensuremath{\overline{\F}}}

\newcommand{\M}{\ensuremath{\mathcal{M}}}
\newcommand{\V}{\ensuremath{\mathcal{V}}}
\renewcommand{\L}{\ensuremath{\mathcal{L}}}

\newcommand{\diag}{\ensuremath{\operatorname{diag}}}
\newcommand{\vvmf}[1]{\ensuremath{M_k(#1)}}
\newcommand{\vvcf}[1]{\ensuremath{S_k(#1)}}
\newcommand{\vvnf}[2]{\ensuremath{S_k(#1)^{\text{new}, #2}}}
\newcommand{\vvof}[2]{\ensuremath{S_k(#1)^{\text{old}, #2}}}
\newcommand{\vvnfwo}[1]{\ensuremath{S_k(#1)^{\text{new}}}}
\newcommand{\vvofwo}[1]{\ensuremath{S_k(#1)^{\text{old}}}}

\newcommand{\pform}[1]{\ensuremath{\bilform{#1}_p}}
\usepackage{authblk}

\title{Almost every vector valued modular form is an oldform}
\author{Fabian Werner}
\affil{werner@mathematik.tu-darmstadt.de}
\begin{document}

\maketitle 
\begin{abstract}
\noindent In this article we show that 'most' of the vector valued modular forms
w.r.t. the Weil representation on the groups rings $\C[D]$ of discriminant forms $D$ are oldforms.
The precise meaning of oldform is that the form can be represented as a sum of lifts 
of vector valued modular forms on group rings of quotients $H^\bot/H$ for isotropic subgroups $H$ of $D$.
In this context, 'most' means that all forms are oldforms if there is a part $\Z/p^e\Z$ 
inside a $p$-part of $D$ that is repeated several times (i.e. $\geq 4,5,7,9$ depending on $p$ and $e$).
We will proceed by giving an oldform detection mechanism.
This criterion also gives rise to an efficient algorithm for computing the decomposition of 
cusp forms into their spaces of old- and newforms when only given the Fourier 
coefficients of a basis of the space of cusp forms.
\end{abstract}

\section{Overview}
Let $\D = (D, Q)$ be a discriminant form. To simplify the exposition, we assume throughout that the 
signature of $D$ is even. All results carry over naturally to the case of an odd signature 
and the metaplectic cover.
Recently, lifts for vector valued modular forms for quotients $H^\bot/H$ -- where $H$ is an isotopic 
subgroup -- of the form
$$G = \sum_{\a \in D_H} G_\a \e_\a \mapsto G\!\uparrow_H\, := \sum_{\gamma \in H^\bot} G_{\gamma + H}\, \e_\gamma$$
\noindent have gained attention. Note that in the body of the paper, 
this map will be called $\uparrow_H^\text{init}$ in order to distinguish 
it from its purely algebraic version (see Section \ref{sec:separation}). 
This map is expected to replace the lifting process for 
dividing levels in the scalar valued case, thus giving 
rise to an oldform/newform theory.
Following the scalar valued ideas, one defines the space of oldforms as
$$\sum_{H ~\text{isotropic}} \vvcf{H^{\!\bot}\!\!/\!H}\!\uparrow_H$$
where $\vvcf{\D}$ is the space of entirely holomorphic vector valued modular cusp forms 
of weight $k$ for the discriminant form $\D$.

\noindent The main purpose of this paper is to show the following:

\vspace{1cm}
\noindent \emph{
If $N\in \N$ is fixed and $\D$ is a discriminant form of level $N$ with 
$|D| \geq N^9$, then every vector valued modular form for $D$ is an oldform. 
This bound ($N^9$) is absolutely not optimal.
}

\vspace{1cm}

\noindent This is stated as Cor. \ref{cor:main-cor} in section \ref{sec:main-theorem}. 
In other words, the meaning of this result is that for every fixed 
level $N$ we only have to study finitely many vector valued 
modular forms for discriminant forms of this level.
In fact, the number $N^9$ is absolutely not optimal, it suffices 
if a certain $p$-part of the discriminant form is repeated often enough, see Thm \ref{thm:prank-geq-oldform}.

\noindent We achieve this by giving a purely algebraic characterization for detecting oldforms 
which is interesting in its own right.
More precisely, we show in Thm \ref{thm:oldforms:detectionThm} that for an arbitrary selection of 
isotropic subgroups $H_1, ..., H_n$,
$$ F ~\text{is an oldform with respect to the $H_1, ..., H_n$} \iff \ker(\downarrow) \subset \ker(\F)$$
where $\downarrow$ is a certain 'algebraic part' of the 'converse' map of $\uparrow$ and
$\F$ is the $\C$-linear map sending $\e_\gamma$ to its component $F_\gamma$ as a modular form for $\Gamma(N)$.
This result can be understood as a generalization of the work of Bruinier, cite{bruinier-converse} Thm. 3.6.

\noindent Moreover, we present the following:
given a basis of vector valued modular forms for some concrete weight and 
discriminant form up to some Sturm bound (which have been created using a computer 
algebra system for example) this characterization allows us to compute bases 
for the spaces of oldforms (with respect to any selection of isotropic subgroups) and, 
if we restrict ourselves to cusp forms, we can compute its orthogonal complement, 
i.e. the space of newforms with a little trick as well, see Thm. \ref{thm:oldforms-summarize}. 
This is great for doing 
concrete experiments with vector valued modular forms, especially in view of the 
fact that M. Raum has recently given an algorith that computes bases of vector 
valued modular forms, see \cite{raum}.
Secondly, more abstractly, the same strategy as in the proof of the theorem above 
allows us to solve the converse problem, i.e. the question of 
whether a vector valued modular form for the smaller discriminant form $H^{\!\bot}\!/H$ is 
a 'down' lift of some form above.

\noindent The paper is organized as follows: 

\noindent \begin{tabular}{lp{10cm}}
  Section & Contents \\
	\hline\\
	Sec. \ref{sec:prerequisites} & 
	  Basic terminology, lattices, discriminant forms, Weil representations, 
		scalar valued and vector 
		valued modular forms, up/down maps \\\\
	Sec. \ref{sec:separation} & 
	  Separation of up/down maps into algebraic and trivial part,
		algebraic parts are homomorphisms of Weil representations\\\\
	Sec. \ref{sec:characterization} &
	  Characterization: $F$ is old iff. $\ker(\downarrow) \subset \ker(\F)$ \\\\
	Sec. \ref{sec:algorithm} &
	  An algorithm for splitting vector valued cusp forms 
		into old and new spaces\\\\
	Sec. \ref{sec:preparation} &
	  Technical lemmas needed for the proof of the main theorem \\\\
  Sec. \ref{sec:main-theorem} &
		Proof of the main theorem, the version announced above is then a corollary
\end{tabular}

\section{Prerequisites}
\label{sec:prerequisites}

The group $\GL^+_2(\R) = \{M \in \Z^{2 \times 2} | \det(M) > 0\}$ acts from the left on the upper half plane
$\H := \{ \tau \in \C | \Im(\tau) > 0\}$ by
$$M.\tau = \frac{a \tau + b}{c \tau + d}, ~~~ M = \mat{a & b \\ c & d}$$
We can continue this action to $\H \cup \R \cup \{\infty\}$ by putting
$M.\infty := a/c$ and if $c\tau+d = 0$ then we put $M.\tau := \infty$.
We also put $(M:\tau) := c \tau + d$. This induces a right action on functions from $\H$ to $\C$ by
$$f|_k M := (M:\tau)^{-k} f(M\tau), ~~~ k \in \Z$$
Mostly, $k$ will be fixed throughout, so we will drop it from the notation.
If $\Gamma$ is a group, a group homomorphism $\chi : \Gamma \to \C$ is called a character.
To simplify the exposition, let us assume that $\Gamma$ is one of the well studied subgroups,
\begin{align*}
  \Gamma_0(N) &:= \left\{ M = \mat{a & b \\ c & d} \in \SL_2(\Z):  c \equiv 0 \!\!\!\!\mod N \right\} \\
  \Gamma_1(N) &:= \left\{ M = \mat{a & b \\ c & d} \in \SL_2(\Z):  a \equiv d \equiv 1 \!\!\!\!\mod N,~ c \equiv 0 \!\!\!\!\mod N \right\} \\
  \Gamma(N) &:= \left\{ M \in \SL_2(\Z) :  M \equiv \text{Id} \!\!\!\!\mod N \right\}
\end{align*}
of $\SL_2(\Z)$. We will only need these subgroups in the body of the paper anyhow.

\begin{dfn}
Let $\Gamma$ be one of the subgroups as above. Let $\chi : \Gamma \to \C^\times$ be a character
with the property that $\Gamma(N) \subset \ker(\chi)$.
An entirely holomorphic modular form of weight $k \in \Z$ for $\Gamma$ with 
character $\chi$ is a function $f : \H \to \C$ such that
\begin{enumerate}
	\item $f$ is holomorphic.
	\item $f|\gamma = \chi(\gamma) f$ for all $\gamma \in \Gamma$.
	\item for every $M \in \SL_2(\Z)$ (not merely all $M \in \Gamma$!), 
	      $f|M(\tau)$ is bounded when $\tau \to \infty$.
\end{enumerate}
The $\C$-vector space of all these functions will be denoted by $M_k(\Gamma, \chi)$.

\noindent In case that $f \in M_k(\Gamma, \chi)$, one can show 
(\cite{diamond-shurman} pp. 1-5, \cite{werner-bsc}, Thm. 2.4.7 or any other 
book on modular forms) that $f$ possesses a Fourier expansion
$$ f(\tau) = \sum_{n = 0}^\infty a_n q^{n/N}, ~~ q = \exp(2 \pi i \tau) = e(\tau)$$
or, more general, for every $M \in \SL_2(\Z)$,
$$ f|_M(\tau) = \sum_{n = 0}^\infty a^{(M)}_n q^{n/N}$$

\noindent $f$ is called cusp form if $a^{(M)}_0 = 0$ for all $M \in \SL_2(\Z)$. The subspace of all such functions will be denoted by $S_k(\Gamma, \chi)$.
\end{dfn}

\noindent We want to describe one possible generalization of the theory of scalar valued modular forms, namely vector valued modular forms. These are functions that 'behave well' under slashing with matrices in $\SL_2(\Z)$ but now they map $\H$ to certain finite dimensional $\C$-vector space. Before we are going to describe the structure we need some terminology.

\noindent Let throughout $R$ be a commutative ring and let $X,Z$ be $R$-modules. 
If $R$ is an integral domain then let $F$ denote its field of fractions.
A $Z$-valued bilinear form is a map
$$b : X \times X \to Z$$
such that $b(\cdot, x) : X \to Z$ and $b(x, \cdot) : X \to Z$ are $R$-linear for all $x \in X$.
If $Z=R$ then we call $b$ integral.
$b$ is called non-degenerate (resp. unimodular) if 
the maps $x \mapsto b(\cdot, x)$ and $x \mapsto b(x, \cdot)$ 
(from $X$ to $\operatorname{Hom}_R(X \to Z)$) are injective (resp. bijective).
We say that $b$ is symmetric if $b(x,y) = b(y,x)$ for all $x,y \in X$.
We say that $b$ is even, if $b(x,x) \in 2R$ for all $x \in X$. 
If $b$ is symmetric and $E,W \subset X$ are submodules, then we
write $E \bot W$ if $b(x,y) = 0$ for all $x \in E, y \in W$.
If $E \bot W$ and for every $s \in E + W$, the elements $e,w$ in $s = e+w$ are unique, 
then we write $E \orthplus W$.

\noindent A $Z$-valued quadratic form is a map $Q : X \to Z$
with the properties that $Q(ax) = a^2Q(x)$ for all $x \in X, a \in R$ and
$$ b_Q : X \times X \to Z, ~~(x,y) \mapsto Q(x+y) - Q(x) - Q(y)$$
is a bilinear map.
$Q$ is called integral if $Z=R$.
$Q$ is called non-degenerate (respectively unimodular) if $b_Q$ has the respective property. 

\noindent A $Z$-valued $R$-lattice is a touple $\L = (L, b)$ consisting
of a freely, finitely generated $R$-module $L$ (i.e. there is a finite set 
$v_1, ..., v_n$ such that $L = Rv_1 \oplus ... \oplus Rv_n$ meaning that every 
$v \in L$ can be written as a unique $R$-linear combination of the $v_i$)
together with a $Z$-valued bilinear form $b$. 
$\L$ is called integral, even, non-degenerate or unimodular if $b$ has the respective property.

\noindent Together with every integral lattice over an integral domain $R$ comes 
its $F$-vector space $V = L \otimes F$ and the $F$-valued 
bilinear form $b_F = b \otimes \id_F$.
For a lattice $\L$ we define $\L'$ to be the touple consisting of
$$ L' := \{ v \in V : b(v,l) \in R ~~ \forall l \in L\}$$
together with the $F$-valued bilinear form $b_F|_{L' \times L'}$ and call 
this the dual lattice to $\L$.

\noindent A discriminant form is a touple $\D = (D, Q)$ consisting of a finite abelian group $D$ and
a so-called finite quadratic form, that is, a non-degenerate quadratic form $Q : D \to \Q/\Z$ such 
i.e. for the associated bilinear form $(\gamma, \delta) = Q(\gamma+\delta) - Q(\gamma) - Q(\delta)$ we have $D^\bot = \{0\}$.
Two discriminant forms $(D, Q), (\tilde{D}, \tilde{Q})$ are called isomorphic if there exists
a $\Z$-module isomorphism $\phi : D \to \tilde{D}$ such that
$\tilde{Q}(\phi(\gamma)) = Q(\gamma)$ for all $\gamma \in D$.

\noindent Mostly, we are rather sloppy with the notation and just write $L$ for $\L$ and $D$ for $\D$ because 
the bilinear, resp. quadratic form will be fixed or clear from the context.

\noindent One of the key-features of discriminant forms is the following:
\begin{thm}
\label{thm:jordan}
Every discriminant form $\D = (D, Q)$ possesses a so-called Jordan splitting, i.e. one finds
a basis in the sense of finitely generated abelian groups of $D$ such that the matrix 
consisting of the bilinear pairings (modulo \Z) is diagonal on the odd $p$-parts 
and almost diagonal on the $2$-adic part. More precisely:
$D$ is the orthogonal sum over components $C$ of the form
  \begin{enumerate}
    \item 
    $C \cong \Z_{p^e}$ for some odd prime $p$ and $C$ is generated by a single 
    element $\gamma$ with $(\gamma, \gamma) = \tfrac{a}{p^e}$ where $a \in \Z, \gcd(a,p)=1$
    and $Q(\gamma) = \tfrac{2^{-1}a}{p^e} + \Z$ where the inversion of $2$ 
    takes place in $\Z_{p^e}$.
    \item 
    $C \cong \Z_{2^e}$ is generated by a single 
    element $\gamma$ with $(\gamma, \gamma) = \tfrac{a}{2^e}$ where $a \in \Z, \gcd(a,2)=1$
    and $Q(\gamma) = \tfrac{a + v2^e}{2^{e+1}} + \Z$ where $v$ is either $0$ or $1$.
    \item 
    $C \cong \Z_{2^e} \times \Z_{2^e}$ is generated by two elements $\gamma, \delta$
    such that the Gram matrix of pairings of $\gamma$ and $\delta$ is given by
      $$2^{-e}\mat{x & 1 \\ 1 & x}$$
    where $x$ is either $0$ or $2$.
    If $x=0$ then $Q(\gamma) = Q(\delta) = 0 + \Z$. We say that this is a block of type (A).
    If $x=2$ then $Q(\gamma) = Q(\delta) = \tfrac{1}{2^e} + \Z$. We say that this is a block of type (B).
  \end{enumerate}
\end{thm}
\begin{proof}
A proof can be found in \cite{werner-jordanform}.
\end{proof}

\noindent Examples of discriminant forms can be obtained by using even, 
non-degenerate $\Z$-lattices $\mathcal{L} = (L, b)$:
As $L$ is even, we can define an integral quadratic 
form $Q : L \to \Z, Q(x) := b(x,x)/2$. Its associated bilinear form 
$b_Q$ is nothing else than $b$, so $b(x,y) = Q(x+y) - Q(x) - Q(y)$ is integral.
Hence, $L \subset L'$ and it makes sense to consider 
$D := L'/L$. Then $Q(x + L) := Q(x) + \Z$ gives a discriminant form.
One can show that in fact, all discriminant forms arise in such a way.
One uses the following strategy: using Thm. \ref{thm:jordan}, we obtain 
a Jordan splitting of $D$. Then one only needs to show the existence 
of a lattice for the Jordan constituents and this problem can be solved
in a surprisingly easy way, see \cite{wall}, Thm 6, mainly p. 297.

\noindent Let $L$ be an even lattice.
For $\tau \in \C$, let $e(\tau) := \exp(2 \pi i \tau)$.
Milgrams formula (\cite{milnor-husemoller} Appendix 4) shows that
$$\sum_{\gamma \in L'/L} Q(\gamma) = \sqrt{|D|} e\left(s/8\right)$$
where $s$ is the signature (over $\R$) of the lattice $L$.
For this reason, for a given discriminant form $D$ there is a number 
$\overline{s} \in \Z/8\Z$ such that 
all even non-degenerate integral $\Z$-lattices $L$ having $D$ as 
their discriminant form (up to isomorphy) have a signature congruent 
to $\overline{s}$ modulo $8$. This element is called the signature $\sign(\D)$ of $\D$.

If the signature of $\D$ is even, then there is a unitary representation $\rho$ of $\SL_2(\Z)$ on
$\C[D]$, the \C-vector space of dimension $|D|$ with canonical 
orthonormal basis $(\e_\gamma)_{\gamma \in D}$. This representation works as follows:
$\SL_2(\Z)$ is generated by $S = \tmat{0 & -1 \\ 1 & 0}$ and $T = \tmat{1 & 1 \\ 0 & 1}$.
Their \C-linear actions are given by
\begin{align*}
  \rho(T) \e_\gamma &= e(Q(\gamma)) \e_\gamma \\
	\rho(S) \e_\gamma &= c_\D \sum_{\beta \in D} e(-(\gamma, \beta)) \e_\beta \\
\end{align*}
where 
$$ c_\D := \frac{e(-\sign(\D)/8)}{\sqrt{|D|}}$$
Of course, no one came up with such formulae out of nowhere, they are conrete 
instances of a more general construction due to A. Weil.
One can see some traces of this process: The action of $S$ is essentially a Fourier transform.
The construction in full generality can be found in \cite{weil}. 
A down-to-earth proof that just makes use of Milgrams Formula (which in turn is proved 
in a down-to-earth way in \cite{milnor-husemoller} Appendix 4) is written down in
\cite{werner-weil}.

\noindent If the signature is not even, then still, there is a representation but one has to pass to a 
degree $2$ metaplectic cover of $\SL_2(\Z)$. We are going to skip this case for the sake of 
readability but all the results carry over naturally.

Let $\D$ be a discriminant form of even signature and $\rho$ the Weil representation on $\C[D]$.
A holomorphic vector valued modular form of weigth $k \in \Z$ is a holomorphic function 
$F : \H \to \C[D]$ satisfying 
$$F\left( \frac{a \tau + b}{c \tau + d}\right) = (c \tau + d)^k \rho\mat{a & b \\ c & d} F(\tau)$$
for all $\tau \in \H, \tmat{a & b \\ c & d} \in \SL_2(\Z)$, in short: $F|M = \rho(M) F$, and 
is holomorphic at $\infty$ meaning that, for example, $F(\tau)$ stays bounded as $\tau \to \infty$.
The set of all such functions will be denoted by $\vvmf{\D}$.

The level $N$ of a discriminant form $\D$ is the smallest natural number $m$ such 
that $mQ(\gamma) = 0 + \Z$ for all $\gamma \in D$.
It is widely believed to be proven for a long time that $\rho(M) = \text{Id}_{\C[D]}$ 
for all $M \in \Gamma(N)$. However, to the best of the authors knowledge, the only 
proof that was officially, completely written down is due to S. Zemel 
(\cite{zemel-diss}, Thm. 3.2) in 2011. An alternative 
down-to-earth proof is due to N.-P. Skoruppa. 
Unfortunately his book about the Weil representation is not published yet.
As $F$ -- or $F|M$, which is just a linear combination of the components 
of $F$ -- stays bounded when $\tau \to \infty$, the same is true for 
every component. Hence, every component of a vector valued modular form 
is a scalar valued modular form for $\Gamma(N)$. We use this assertion 
without mentioning it any further. $F$ is called a cusp form iff. 
every component is a cusp form. The subspace of cusp forms will be 
written as $\vvcf{\D}$.

We describe the Petersson scalar product:
The measure
$$\nu(M) = \int_M \frac{1}{y^2}dx dy,~~ M \subset \H ~\text{Lebesgue-measurable}$$
is $\GL_2^+(\R)$-invariant (see \cite{koecher-krieg}, Kap. IV, \textsection 3).
Let $\mathcal{A}$ be an arbitrary fundamental domain for $\Gamma$,
that is, a 'nice' system of representatives for $\Gamma\backslash\H$
with the property that $\nu(\partial\mathcal{A}) = 0$ where
$\partial \mathcal{A}$ is the topological boundary of $\mathcal{A}$.
Different authors give different (wrong!, see \cite{elstrodt}) 
definitions of 'nice' and forget about the 
additional condition. However,
for the three subgroups $\Gamma_0(N), \Gamma_1(N), \Gamma(N)$,
every of the definitions floating around in current literature 
(for example: \cite{miyake}, \textsection 1.6, \cite{koecher-krieg} Kap. II \textsection 3)
together with the condition $\nu(\partial\mathcal{A}) = 0$ will do.

\noindent Let $f,g \in S_k(\Gamma)$ with $\Gamma$ being a 
'nice' subgroup of $\SL_2(\Z)$, say one of the examples given above.

\noindent The map 
$$\ideal{f,g} := \frac{1}{[\SL_2(\Z) : \Gamma]}\int_{\mathcal{A}} f(\tau) \overline{g}(\tau) y^k dx dy/y^2$$
is convergent (\cite{koecher-krieg} Kap. IV, \textsection 3, 
\cite{diamond-shurman} \textsection 5.4, etc.),
is independent of the chosen fundamental domain 
(\cite{koecher-krieg}, Kap. IV, \textsection 3, pp. 231-232)
and turns $S_k(\Gamma)$ into a Hilbert space.
In complete analogy we define the Petersson scalar product for vector valued cusp forms $F, G$ to be
$$\ideal{F, G} := \int_{\A} \sum_{\gamma \in D} F_\gamma \overline{G}_\gamma y^k dx dy/y^2$$
where here, $\mathcal{A}$ is a fundamental domain for $\SL_2(\Z) \backslash \H$.

\noindent We are going to describe a part of the theory of (scalar valued) modular forms called 'old/newform-theory':
Let $f \in S_k(\Gamma_0(A))$ for some $A | B$. Then there are multiple ways of interpreting $f$ 
as an element in $S_k(\Gamma_0(B))$. Generally, one can consider $f \mapsto f(c\tau)$ 
where $c | \tfrac{A}{B}$; $c=1$ corresponding to the trivial inclusion 
$S_k(\Gamma_0(A)) \subset S_k(\Gamma_0(B))$. The span of the images of all these maps
is called the subspace of oldforms. The reason why we restrict ourselves to cusp 
forms is that we want to take the orthogonal complement w.r.t. the Petersson scalar product of 
the space of oldforms and call this the space of newforms. This space has some extremely important 
properties (eigenbasis for Hecke operators, Euler products, connections to elliptic curves, etc), 
see for example \cite{miyake}, \textsection 4.6 ff. Hence, it is important to ask whether 
there is a similar construction for vector valued modular forms.

Let $\D = (D, Q)$ be a discriminant form.
An element $\gamma$ of $D$ is called isotropic if $Q(\gamma) = 0 + \Z$. 
A subgroup $H$ of $D$ is called isotropic if all elements of $H$ are isotropic.
If $H$ is an isotropic subgroup we put
$D_H := H^\bot / H$. Then, $\D_H := (D_H, Q_H)$ with $Q_H(\gamma + H) := Q(\gamma)$ becomes
a discriminant form and satisfies $\sign(\D_H) = \sign(\D)$ and $|D_H| = |D|/|H|^2$.
(The proof of this assertion is left to the reader).
When isotropic subgroups $H_1, ..., H_n$ are given we just write $\D_i$ in place of $\D_{H_i}$
and $D_i$ in place of $H_i^\bot/H_i$.

\noindent Recently, operators of the form
\begin{equation}
\uparrow^\text{init}_H : \vvmf{\D_H} \to \vvmf{\D}, ~~
     \sum_{\a \in D_H} G_\a \e_\a \mapsto G\!\uparrow^\text{init}_H \,:= \sum_{\gamma \in H^\bot} G_{\gamma + H}\, \e_\gamma
\label{eq:initUp}
\end{equation}
have gained attention. 
Abstractly, these operators are expected to replace the lifting process for 
dividing levels in the scalar valued case, hence giving rise to a vector valued oldform/newform theory.
They have been used for example, to study in which cases 
certain orthogonal modular forms arise as Borcherds lifts (see \cite{bruinier-converse}) 
and under which conditions a vector valued modular form is induced by a scalar valued one (see \cite{sch:weil2}).
There is also a 'converse' map:
\begin{equation}
\downarrow^\text{init}_H: \vvmf{\D} \to \vvmf{\D_H}, ~~
     \sum_{\gamma \in D} F_\gamma \e_\gamma \mapsto F\!\!\downarrow^\text{init}_H\, := \sum_{\a \in D_H} \left(\sum_{\gamma \in \a} F_\gamma\right) \e_\a
\label{eq:initDown}
\end{equation}

\noindent (Remark that it is not clear that these operators really map vector valued modular forms 
to vector valued modular forms again; we will prove it in the next section). 
We write them with a superscript 'init' for 'initial' in order not to confuse them 
with their 'algebraic' parts, see Sec. \ref{sec:separation}.

Following the ideas in the scalar valued case we define old- and newforms:
Take isotropic subgroups $H_1, ..., H_k$ of $D$. We define the space 
of vector valued oldforms w.r.t. $H_1, ..., H_k$ to be
$$\vvof{\D}{H_1, ..., H_k} := \vvcf{\D_1}\!\uparrow^\text{init}_{H_1} + ... + \vvcf{\D_k}\!\uparrow^\text{init}_{H_k}$$
Analogously, the space of newforms is
$$\vvnf{\D}{H_1, ..., H_k} := \left(\vvof{\D}{H_1, ..., H_k}\right)^\bot$$
where the orthogonal complement is taken with respect to the Petersson scalar product for vector valued modular forms.

\section{Separarion of the up and down maps}
\label{sec:separation}
In this section we will separate the up and down maps from the introduction into two parts. 
In order not to confuse them we will call the up/down maps on modular forms $\uparrow^{\text{func}}_H$ 
and $\downarrow^{\text{func}}_H$ if in doubt.
It turns out that they can be written as $\uparrow^{\text{func}}_H = \uparrow_H \otimes \id$ for 
some \C-linear, purely algebraic map $\uparrow$ (and similarly with $\downarrow$). 
These maps are of crucial importance for the characterization of oldforms.

\noindent Generally speaking, given vector spaces $V, V_1, ..., V_n$ and vector space homomorphisms
$\alpha_i : V \to V_i$ then we denote by $\alpha := (\alpha_1 , ... , \alpha_n)$ the homomorphism 
from $V$ to $V_1 \oplus ... \oplus V_n$ given by 
$\alpha(v) := (\alpha_1(v), ..., \alpha_n(v))$.
Conversly, given vector spaces $V, V_1, ..., V_n$ and vector space homomorphisms
$\beta_i : V_i \to V$ then we denote by $\beta := \beta_1 + ... + \beta_n$ the homomorphism 
from $V_1 \oplus ... \oplus V_n$ to $V$ given by 
$\beta(v_1, ..., v_n) := \beta_1(v_1) + ... + \beta_n(v_n)$.

\noindent It is easy to see that
\begin{rmk}
\label{rmk:oldforms:sumsHoms}
If $\phi, \phi_1, ..., \phi_n$ are representations of some group $G$ 
on $V, V_1, ..., V_n$. We endow $V_1 \oplus ... \oplus V_n$ with the 
representation $\phi_1 \oplus ... \oplus \phi_n$.
\begin{enumerate}[(a)]
  \item If all $\alpha_i$ are homomorphisms of representations then so is $\alpha = (\alpha_1, ..., \alpha_n)$
  \item If all $\beta_i$ are homomorphisms of representations then so is $\beta = \beta_1 + ... + \beta_n$.
\end{enumerate}
\end{rmk}

\begin{dfn}
\label{dfn:oldforms:algMaps}
Let $\D = (D, Q)$ be a discriminant form and $H$ be an arbitrary isotropic subgroup.
We let $\pi : H^\bot \to D_H$ denote the projection $\pi(\gamma) = \gamma + H$ and
we put $\downarrow_H : \C[D] \to \C[D_H]$ to be the \C--linear map
$$\downarrow_H \bigg(\sum_{\gamma \in D} c_\gamma \e_\gamma \bigg) 
    := \sum_{\a \in D_H} \bigg(\sum_{\gamma \in \pi^{-1}(\a)} c_\gamma \bigg)~~\e_\a$$
Further we define a \C--linear map $\uparrow_H : \C[D_H] \to \C[D]$ as
$$\uparrow_H\bigg(\sum_{\a \in D_H} c_\a \e_\a\bigg) 
    := \sum_{\gamma \in H^\bot} c_{\gamma + H} \e_\gamma$$

\noindent When isotropic subgroups $H_1, ..., H_n$ are given we just write $\D_i$ in place of $\D_{H_i}$,
$\downarrow_i$ in place of $\downarrow_{H_i}$ and similarly with the up arrow maps
and we let 
\begin{align*}
  \downarrow_{H_1, ..., H_n}              &:= (\downarrow_1, ..., \downarrow_n)\\
  \uparrow_{H_1, ..., H_n}                 &:=\, \uparrow_1 + ... + \downarrow_n\\
\end{align*}
and drop the $H_i$ from the notation as they will be clear from the context.
If no isotropic subgroups are given, then we just write $\uparrow$, $\downarrow$ 
as associated (in the sense above) to all isotropic subgroups in $D$.
\end{dfn}

\begin{lem}
\label{lem:oldforms:vIsHom}
Let $\D = (D, Q)$ be a discriminant form. Let $\rho : \SL_2(\Z) \to \C[D]$ denote the Weil representation of $\D$.
\begin{enumerate}[(a)]
  \item 
    Let $H$ be an arbitrary isotropic subgroup and let $\eta$ be the Weil rep. of $\D_H$.
    Then the maps $\downarrow_H, \uparrow_H$ are homomorphisms of the Weil representations, i.e.
          \[ 
            \xymatrix{ \C[D]\ar[rr]^{\rho(M)}  \ar@{<->}[d]_{\downarrow_H}^{\uparrow_H}     & & \C[D] \ar@{<->}[d]^{\downarrow_H}_{\uparrow_H} \\
                       \C[D_H]\ar[rr]^{\eta(M)}        &  & C[D_H] } 
          \]
        commutes for every $M \in \SL_2(\Z)$.
        \label{lem:oldforms:vIsHom:alghoms}
  \item 
	Let $H_1, ..., H_n$ be arbitrary isotropic subgroups of $D$ and let $\rho_i$ be the Weil representation of $\D_i$ for $i=1,...,n$.
	Then $\downarrow$ and $\uparrow$ are homomorphisms of representations, i.e.
          \[ 
            \xymatrix{ \C[D]\ar[rr]^{\rho(M)}  \ar@{<->}[d]_{\downarrow}^{\uparrow}     & & \C[D] \ar@{<->}[d]^{\downarrow}_{\uparrow} \\
                       \C[D_{H_1}] \oplus ... \oplus \C[D_{H_n}]\ar[rr]^{(\rho_1 \oplus ... \oplus \rho_n)(M)}        &  & \C[D_{H_1}] \oplus ... \oplus \C[D_{H_n}] } 
          \]
        commutes for every $M \in \SL_2(\Z)$.
        \label{lem:oldforms:vIsHom:sumhoms}
\end{enumerate}
\end{lem}

\begin{proof}
\noindent \eqref{lem:oldforms:vIsHom:alghoms}:
We need to show that for all $x \in \C[D]$ and all $M \in \SL_2(\Z)$,
  $$\eta(M)\downarrow_H\!\!(x) =\,\, \downarrow_H\!\!(\rho(M)x)$$
Since all maps $\downarrow_H, \uparrow_H, \rho(M), \eta(M)$ are $\C$-linear, 
it suffices to show the assertion for $x = \e_\gamma$.
Since $\SL_2(\Z)$ is generated by $S, T$, and both, $\rho, \eta$ are left actions, it suffices to show the assertion
for $M=S, M=T$. 

\noindent On $M=T$:
\begin{align*}
\downarrow_H\!\!(\rho(T)\e_\gamma) &=\,\, \downarrow_H\!\!(e(Q(\gamma)) \e_\gamma)
   = e(Q(\gamma)) \downarrow_H\!\!(\e_\gamma)
	 = e(Q(\gamma)) \e_{\gamma + H} \\
	&= e(Q_H(\gamma + H)) \e_{\gamma + H}
   = \eta(T) \e_{\gamma + H}
	 ~\!= \eta(T) \downarrow_H\!\!(\e_\gamma)
\end{align*}

\noindent On $M=S$:
we write 
\begin{equation}
\e_\gamma = \sum_{\delta \in \D} c_\delta \e_\delta ~\text{with}~ c_\delta = \one_{\gamma = \delta}
\label{eq:oldforms:vIsHom:alghoms:xegamma}
\end{equation}
then
\begin{align*}
  \downarrow_H\!\!(\rho(S) \e_\gamma) 
    &=\, \downarrow_H\!\!(c_\D \sum_{\mu \in D} e(-\gamma, \mu) \e_\mu) \\
    &= c_\D \sum_{\mu \in D} e(-\gamma, \mu) \downarrow_H\!\!(\e_\mu) \\
    &= c_\D \sum_{\mu \in D} e(-\gamma, \mu) \sum_{\a \in \D_H} \left(\sum_{\lambda \in \a} c_\lambda \right) \e_\a \\
    &\stackrel{\eqref{eq:oldforms:vIsHom:alghoms:xegamma}}{=} 
		  c_\D \sum_{\mu \in D} e(-\gamma, \mu) \sum_{\a \in \D_H} \left(\sum_{\lambda \in \a} \one_{\lambda=\mu} \right) \e_\a \\
    &= c_\D \sum_{\mu \in D} e(-\gamma, \mu) \sum_{\a \in \D_H} \one_{\mu \in \a} \e_\a \\
    &= c_\D \sum_{\a \in \D_H} \left(\sum_{\mu \in \a} e(-\gamma, \mu) \right) \e_\a
\end{align*}
Let us select a fixed representative $\a_0 \in \a \in D_H$ for every class. Then this expression can be rewritten to
\begin{align*}
  &= c_\D \sum_{\a \in \D_H} \left(\sum_{h \in H} e(-\gamma, \a_0 + h) \right) \e_\a \\
  &= c_\D \sum_{\a \in \D_H} e(-\gamma, \a_0) \left(\sum_{h \in H} e(-\gamma, h) \right) \e_\a
\end{align*}

\noindent In the case that $\gamma \notin H^\bot$, the map $\chi : \mu \mapsto e(-\gamma, \mu)$ is a 
nontrivial character of the group $H$. As for every nontrivial character 
$\psi$ of a finite group $A$, we have $\sum_{a \in A} \psi(a) = 0$, the expression just evaluates to $\sum 0 = 0$.
This coincides with $\eta(S)\downarrow_H\!\!(x)$ as $\downarrow_H\!\!(\e_\gamma) = 0$ in this case as well.
Now let $\gamma \in H^\bot$. Then the character $\chi$ is trivial and we can continue
\begin{align*}
    &= c_\D \sum_{\a \in \D_H} e(-\gamma, \a_0) \left(\sum_{h \in H} e(-\gamma, h) \right) \e_\a \\
    &= c_\D \sum_{\a \in \D_H} e(-\gamma, \a_0) |H| \e_\a \\
    &= |H|c_\D  \sum_{\a \in \D_H} e_H(-\gamma + H, \a_0 + H) \e_\a \\
    &= |H|c_\D  \sum_{\a \in \D_H} e_H(-\gamma + H, \a) \e_\a
\end{align*}
We have $|H|c_\D = c_{\D_H}$:
$$|H|c_\D = \frac{1}{1/\sqrt{|H|^2}} \frac{e(\sign(\D)/8)}{\sqrt{|D|}} = \frac{e(\sign(\D)/8)}{\sqrt{|D|/|H|^2}} = c_{\D_H}$$
as $\sign(\D) = \sign(\D_H)$ and $|D_H| = |D|/|H|^2$. Finally,
$$  \downarrow_H\!\!(\rho(S) \e_\gamma) 
    = c_{\D_H}  \sum_{\a \in \D_H} e_H(-[\gamma + H], \a) \e_\a
		= \eta(S) \e_{\gamma + H} = \eta(S) \downarrow_H\!\!(\e_\gamma)$$

\noindent \eqref{lem:oldforms:vIsHom:sumhoms}:
This follows from Rmk. \ref{rmk:oldforms:sumsHoms} and \eqref{lem:oldforms:vIsHom:alghoms}.

\end{proof}

\noindent One could wonder about the naming convention for our operators $\downarrow_H$.
The similarity to $\downarrow_H^{\text{init}}$ is no coincidence. In fact, our $\downarrow_H$ operators can be seen to be the 'algebraic part' of
the operators as introduced in the section before. We state this more precisely now:

\noindent Vector valued modular forms for $\D$ can be viewed as elements of the more general vector space
$\V_{(\D, k)} := M_k(\Gamma(N)) \otimes \C[D]$, the isomorphism between maps from $\H$ to $\C[D]$ 
(having modular forms in every component) and $\V_{(\D, k)}$ being
$$ \Phi_\D : F = \sum_{\gamma \in D} F_\gamma \e_\gamma \mapsto \sum_{\gamma} F_\gamma \otimes \e_\gamma$$
On $\V_{(\D, k)}$ there are two group actions. Firstly, $\SL_2(\Z)$ acts on $M_k(\Gamma(N))$ 
from the right by the usual slash action $f|_M = (c\tau + d)^{-k} f(\tfrac{a\tau + b}{c\tau + d})$
for $M = \tmat{a&b\\c&d}$. This gives rise to the right action $|^\otimes_M := |_M \otimes\, \id_{\C[D]}$.
The Weil representation $\rho$ is a left action of $\SL_2(\Z)$ on $\C[D]$ which gives rise to 
a the left action $\Psi = \id_{M_k(\Gamma(N))} \otimes \rho$ on $\V_{(\D, k)}$.
The set of vector valued modular forms is now
\begin{align*} 
  \M_{(\D, k)} := \{F \in \V_k :~~ F|^\otimes_M = \Psi(M)F ~\text{and $F$ bounded at $\infty$}\}
\end{align*}
in the sense that $\Phi(\vvmf{\D}) = \M_{(\D, k)}$.

\noindent Our $\downarrow_H, \uparrow_H$ operators give rise to maps 
$\downarrow^\text{func}_H : \V_{(\D, k)} \to \V_{(\D_H, k)}$ and
$\uparrow^{\text{func}}_H : \V_{(\D_H, k)} \to \V_{(\D, k)}$ by putting
$\downarrow^{\text{func}}_H \,:=\, \downarrow_H \otimes \, \id_{M_k(\Gamma(N))}$
and
$\uparrow^{\text{func}}_H \, := \, \uparrow_H \otimes \, \id_{M_k(\Gamma(M))}$
where $M$ is the level of $\D_H$.
We then put
\begin{align*}
  \downarrow^{\text{func}} &:= (\downarrow^{\text{func}}_1, ..., \downarrow^{\text{func}}_n)\\
  \uparrow^{\text{func}}   &:= \,\uparrow^{\text{func}}_1 + ... + \uparrow^{\text{func}}_n
\end{align*}

\noindent Unwinding the definitions, we see that 
$$ \Phi_{(\D_H, k)}(F\!\downarrow_H^{\text{init}}) = \Phi_{(\D, k)}(F)\!\downarrow^\text{func}_H ~~ \text{and} ~~ 
   \Phi_{(\D, k)}(G\!\uparrow_H^{\text{init}}) = \Phi_{(\D_H, k)}(G)\!\uparrow^\text{func}_H
$$
where $\uparrow_H^{\text{init}}, \downarrow_H^{\text{init}}$ are the initial definitions as given in
\eqref{eq:initUp}, \eqref{eq:initDown}. 
So, if we interpret the operators on the right space then the up/down arrows on functions 
are just the algebraic up and down maps tensored with id.
Hence, we will use the superscripts 'func' and 'init' interchangeably.
\noindent We can now see the reason
why $\uparrow^{\text{func}}, \downarrow^{\text{func}}$ map vector valued modular forms to such again: they come from
purely algebraic homomorphisms of Weil representations.

\begin{lem}
\label{lem:oldforms:tensorIsHom}
  \begin{enumerate}[(a)]
  \item The maps $\downarrow^{\text{func}}_H, \uparrow^{\text{func}}_H$ are homomorphisms 
        of the tensored right slash action, i.e.
          \[ 
            \xymatrix{ \V_{(\D, k)}\ar[rr]^{| \otimes \id_{\C[D]}}  \ar[d]_{\downarrow^{\text{func}}_H}     & & \V_{(\D, k)} \ar[d]^{\downarrow^{\text{func}}_H} \\
                       \V_{(\D_H, k)}\ar[rr]^{| \otimes \id_{\C[D_H]}}        &  & \V_{(\D_H, k)} } 
          \]
        commutes for every $M \in \SL_2(\Z)$.
        \label{lem:oldforms:tensorIsHom:funcHoms}
  \item $\downarrow^{\text{func}}(\M_{(\D, k)}) \subset \M_{(\D_H, k)}$ and
        $\uparrow^{\text{func}}(\M_{(\D_H, k)}) \subset \M_{(\D, k)}$.
        \label{lem:oldforms:tensorIsHom:funchoms}
  \end{enumerate}
\end{lem}

\begin{proof}
\noindent \eqref{lem:oldforms:tensorIsHom:funcHoms}
This is trivially true: Let $A$ be the level of $\D_H$ then for 
$F = \sum_\gamma F_\gamma \otimes \e_\gamma$ and $M \in \SL_2(\Z)$ we get
\begin{align*}
  (\downarrow^{\text{func}}_H\!\!(F)) &. [~\big| \otimes \id_{\C[D_H]}](M) \\
   &= \sum_\gamma (F_\gamma \,\otimes \downarrow_H\!\!(\e_\gamma)) . [~\big|_M \otimes \id_{\C[D_H]}] \\
   &= \sum_\gamma (F_\gamma)|_M \,\otimes \downarrow_H\!\!(\e_\gamma) \\
   &= \sum_\gamma \downarrow^{\text{func}}_H\!\!((F_\gamma)|_M \otimes \e_\gamma) \\
   &= \downarrow^{\text{func}}_H \!\!\left(\sum_\gamma (F_\gamma \otimes \e_\gamma).[~\big| \otimes \id_{\C[D]}](M)\right) \\
   &= \downarrow^{\text{func}}_H \!\!\left( \left[\sum_\gamma F_\gamma \otimes \e_\gamma\right].[~\big| \otimes \id_{\C[D]}](M)\right) \\
   &= \downarrow^{\text{func}}_H \!\!\left(F.[~\big| \otimes \id_{\C[D]}](M)\right)
\end{align*}

\noindent \eqref{lem:oldforms:tensorIsHom:funchoms}
For $F = \sum_\gamma F_\gamma \otimes \e_\gamma$ we get
\begin{align*}
  (\eta \otimes \id)(M) & \downarrow^{\text{func}}_H\!\!(F) \\
    &= \sum_\gamma (\id \otimes \eta)(M) [F_\gamma \,\otimes \downarrow_H\!\!(\e_\gamma)] \\
    &= \sum_\gamma [F_\gamma \otimes \eta(M)\downarrow_H\!\!(\e_\gamma)] \\
    &= \sum_\gamma [F_\gamma \,\otimes \downarrow_H \!\rho(M)(\e_\gamma)] 
			  & \text{By Lem. \ref{lem:oldforms:vIsHom}\eqref{lem:oldforms:vIsHom:alghoms}} \\
    &= \sum_\gamma \downarrow^{\text{func}}_H [F_\gamma \otimes  \rho(M)(\e_\gamma)] \\
    &= \downarrow^{\text{func}}_H \sum_\gamma (\id \otimes \rho)(M)[F_\gamma \otimes  \e_\gamma] \\
    &= \downarrow^{\text{func}}_H \Phi(M) F \\
    &\stackrel{\eqref{lem:oldforms:tensorIsHom:funcHoms}}{=} 
       \downarrow^{\text{func}}_H (F.[~\big| \otimes \id](M))\\
    &= (\downarrow^{\text{func}}_H F).[~\big| \otimes \id](M)\\
\end{align*}
\noindent The other inclusion is proved similarly.
\end{proof}

\noindent We have shown that the up and down maps are really well defined, i.e. that the turn vector valued modular forms into such again.

\section{Detecting oldforms}
\label{sec:characterization}

In this section we will state and prove a detection mechanism for vector valued oldforms.

\begin{dfn}
\label{dfn:evaluation-map}
Let $N$ be the level of $D$. For $F = \sum_{\gamma \in D} F_\gamma \e_\gamma \in \vvmf{D}$ 
we define the $\C$-linear map
  $$\F : \C[D] \to M_k(\Gamma(N)), \e_\gamma \mapsto F_\gamma$$
\end{dfn}

\noindent $\F$ can be viewed as a evaluation map.

\noindent The crucial condition for $F$ to be an oldform now is $\ker(\downarrow) \subset \ker(\F)$. 
This simply states that "all relations among the components of $F$ that we could 
expect if $F$ {\bfseries was} an oldform (with respect to the $H_1, ..., H_n$) do really exist",
see the direction $"\Leftarrow"$ in the proof of Thm \ref{thm:oldforms:detectionThm}.

\noindent We recall the following simple lemma from basic representation theory
\begin{lem} [Maschke]
\label{lem:oldforms:maschke}
Suppose $G$ is a finite group, $K$ is a field with
  $$\chr(K) = 0 ~\text{or}~ \gcd(\chr(K), |G|) = 1$$
Let $\rho : G \to \GL(V)$ be a finite dimensional representation over $K$ and let $U \subset V$ 
be a $G$-invariant subspace (meaning that for every $u \in U$, 
$\rho(g)u \in U$ for all $g \in G$), then there exists a complementary 
$G$-invariant subspace, i.e. there exists a subspace $W \subset V$ such that
$V = W \oplus U$ and $W$ is $G$-invariant.
\end{lem}

\noindent An important corollary one can deduce from this lemma is that 
homomorphisms of subspaces can always be continued to the full space:
\begin{cor}
\label{lem:oldforms:complementarySubspace}
Suppose $G$ is a finite group, $K$ is a field with
  $$\chr(K) = 0 ~\text{or}~ \gcd(\chr(K), |G|) = 1$$
Let $\rho : G \to \GL(V), \eta : G \to \GL(W)$ be finite dimensional 
representations over $K$ and let $U \subset V$ be $G$-invariant. Assume further that
$\vartheta : U \to W$ is a $K$-linear homomorphism of 
representations $(U, \rho(G)|_U) \to (W, \eta)$ 
(i.e. we assume $\vartheta(\rho(g)u) = \eta(g) \vartheta(u)$ 
for all $u \in U, g \in G$). Then $\vartheta$ can be continued to
a homomorphism of representations $\Theta : (V, \rho) \to (W, \eta)$.
\end{cor}
\begin{proof}
By Lemma \ref{lem:oldforms:maschke}, we can find a $G$-invariant complement $E$ to $U$. 
For $v \in V = E \oplus U$, i.e. $v = e + u$ we put $\Theta(e+u) := \vartheta(u)$.
Then $\Theta$ continues $\vartheta$ and it is a homomorphism of representations as
$\vartheta$ was and $E$ is $G$-invariant.
\end{proof}

\begin{thm}
\label{thm:oldforms:detectionThm}
Let $\D = (D, Q)$ be a discriminant form of even signature. Let $H_1, ..., H_n$ 
be arbitrary isotropic subgroups and $F \in \vvmf{\D}$, then $F$ is an oldform
with respect to the $H_1, ..., H_n$ if and only if $\ker(\downarrow_{H_1, ..., H_n}) \subset \ker(\F)$.
\end{thm}
\begin{proof}

\noindent 
For brevity we only write $\downarrow$ in place of $\downarrow_{H_1, ..., H_n}$.

\noindent"$\Rightarrow$": 
Let $\rho_1, ..., \rho_n$ be the Weil representations of $\SL_2(\Z)$ on $\C[D_1], ..., \C[D_n]$.
Then we let $\eta := \rho_1 \oplus ... \oplus \rho_n$. This is a representation of $\SL_2(\Z)$ on
$X := \C[D_1] \oplus ... \oplus \C[D_n]$. We identify $X$ with its 
isomorphic copy 
$$X \cong \C\left[ \bigsqcup_{i=1}^{n} D_i\right]$$
i.e. instead of writing elements as touples
$(\a_1, ..., \a_n)$ where $\a_i \in D_i$, we write them all as $\C$-linear combinations of elements
of the form $[i, \a]$ where $\a \in D_i$. We also put $Y := \im(\downarrow)$.

Fix $i \in \{1,...,n\}$ and $\gamma \in D$. Let $\pi_i : H_i^\bot \to D_i$ 
be the natural projection $\pi_i(\mu) = \mu + H_i$. Suppose $\gamma \in H_i^\bot$. 
Then 
$$
  \downarrow_{H_i}\!\!(\e_\gamma) 
	   = \,\downarrow_{H_i} \!\!\bigg(\sum_{\mu \in D} \mathbf{1}_{\gamma=\mu} \e_\mu \bigg) 
     = \sum_{\b \in D_i} \sum_{\mu \in \pi_i^{-1}(\b)} \mathbf{1}_{\mu=\gamma} \e_\b
		 = \e_{\gamma + H_i}
		 = \sum_{\substack{\{ \b \in D_i : \\ \gamma \in \b\}}} \e_{\b}
$$
because $\gamma$ is contained in precisely one class, namely $\gamma + H_i$.
If $\gamma \notin H_i^\bot$ then both sides of the equation give $0$, hence
$$  
  \downarrow_{H_i}\!\!(\e_\gamma) = \sum_{\substack{\{ \b \in D_i : \\ \gamma \in \b\}}} \e_{\b}
$$
holds for all $\gamma \in D$ and all $i=1,...,n$. Consequently,
\begin{equation}
\label{eq:downbasis}
  \downarrow\!\!(\e_\gamma) = \sum_{i=1}^{n} \sum_{\substack{\{ \b \in D_i : \\ \gamma \in \b\}}} \e_{\b}
  ~~~\forall \gamma \in D
\end{equation}

\noindent We use the Assumption in the following way: As $\C[D]/\ker(\downarrow) \hookrightarrow \C[D]/\ker(\FF)$,
we can push the map $\F$ forward to $\C[D]/\ker(\downarrow) \cong \im(\downarrow) = Y$ by setting
$$\FF(y) := \F(\text{arbitrary preimage of $y$ under $\downarrow$ in $\C[D]$})$$
In particular, for every $\gamma \in D$, we have that
$\e_\gamma$ is a preimage of $\downarrow\!\!(\e_\gamma)$, hence
\begin{equation}
\label{eq:oldforms:Fvgamma}
\FF(\downarrow\!\!(\e_\gamma)) = \F(\e_\gamma) = F_\gamma
\end{equation}

\noindent We will need the following lemma:

\begin{lem}
\label{lem:oldforms:getOldForms}
The space $Y=\im(\downarrow)$ is $\SL_2(\Z)$ invariant. 
If $M \in \SL_2(\Z)$ is such that $\rho(M) \in \C^{|D| \times |D|}$ is symmetric, then the
diagram

\[ 
\xymatrix{ Y\ar[rr]^{\FF}  \ar[d]_{\eta(M)}     & & M_k(\Gamma(N)) \ar[d]^{f \mapsto f|_M} \\
           Y\ar[rr]^{\FF}        &  & M_k(\Gamma(N)) } 
\]

\noindent commutes, i.e.
$$\FF(\eta(M)y) = \FF(y)|_M ~~ \text{for all $y \in Y$}$$
\end{lem}
\begin{proof}
It is clear that $Y$ is $\SL_2(\Z)$ invariant, because the map $\downarrow$ is a 
homomorphism of representations by Lemma \ref{lem:oldforms:vIsHom}\eqref{lem:oldforms:vIsHom:alghoms}, i.e.
if $y = \,\,\downarrow\!\!(x)$ then 
$$\eta(M)y = \eta(M)\!\downarrow\!(x) =\,\, \downarrow\!(\rho(M)x) \in \im(\downarrow)$$ 
for all $M \in \SL_2(\Z)$.
Now let $\rho(M) \e_\gamma = \sum_{\mu \in D} c_{\gamma, \mu} \e_\mu$ with 
$c_{\gamma, \mu} = c_{\mu, \gamma}$ by the assumption on the symmetry.
As $F$ is a vector valued modular form, $F_\gamma|_M = \sum_{\mu \in D} c_{\mu, \gamma} F_\mu 
= \sum_{\mu \in D} c_{\gamma, \mu} F_\mu$.
Since all the maps are $\C$-linear, it suffices to show the assertion for the generators $y = \,\,\downarrow\!\!(\e_\gamma)$
Now
\begin{align*}
  \FF(y)|_M &= \FF(\downarrow\!\!(\e_\gamma))|_M = F_\gamma|_M & \text{by \eqref{eq:oldforms:Fvgamma}} \\
    &= \sum_{\mu \in D} c_{\gamma, \mu} F_\mu \\
    &= \sum_{\mu \in D} c_{\gamma, \mu} \FF(\downarrow\!\!(\e_\mu)) & \text{by \eqref{eq:oldforms:Fvgamma}} \\
    &= \FF\left(\sum_{\mu \in D} c_{\gamma, \mu} \downarrow\!\!(\e_\mu) \right) \\
    &= \FF \downarrow\!\!\left( \sum_{\mu \in D} c_{\gamma, \mu} \e_\mu \right) \\
    &= \FF \downarrow\!\!( \rho(M) e_\gamma) \\
    &= \FF \eta(M) \downarrow\!\!(\e_\gamma) & \text{by Lemma \ref{lem:oldforms:vIsHom}} \\
    &= \FF (\eta(M)y) 
\end{align*}
\end{proof}

\begin{rmk}
  The discrepancy (i.e. the reason why we need to assume that $\rho(M)$ is symmetric) is that
  $\eta$ is a left action of $\SL_2(\Z)$ and slashing $f \mapsto f|_M$ is a right action.
\end{rmk}

\noindent We consider the inclusion map $\iota : Y \hookrightarrow Y$. Clearly, as $Y$ is $\SL_2(\Z)$ invariant, 
it makes sense to view $\eta$ as a representation of $\SL_2(\Z)$ on $Y$. Then, $\iota$ is clearly 
a homomorphism of representations. By Lemma \ref{lem:oldforms:complementarySubspace}, we can
continue $\iota$ to a homomorphism of representations
$$\Theta : X \to Y$$
This needs some clarification. Of course, $\SL_2(\Z)$ is not a finite group but
as Weil representations are trivial on $\Gamma(N)$, they can be viewed as representations
of the group $\SL_2(\Z)/\Gamma(N) \cong \SL_2(\Z_N)$ which is finite!
For every $i=1,...,n$ we define
\begin{equation}
\label{eq:defnG}
G_i := \sum_{\b \in D_i} G^{(i)}_\b \e_\b, ~~~ G^{(i)}_\b := \FF(\Theta([i, \b]))
\end{equation}
We claim that $G_i \in \vvmf{\D_i}$:
It suffices to check that 
$G_i$ slashes correctly under $S,T$. For these matrices,
$\rho(S), \rho(T)$ are symmetric and hence, Lemma \ref{lem:oldforms:getOldForms} 
is applicable. Let $M=S$ or $M=T$ (in fact, let $M$ be arbitrary such that $\rho(M)$ and $\rho_i(M)$ are symmetric).
Let $\rho_i(M)\e_\b = \sum_{\c \in D_i} c^{i, M}_{\b, \c} \e_\c$ with $c^{i,M}_{\b, \c} = c^{i,M}_{\c,\b}$. 
We need to see that
$G|_M = \rho(M) G$, i.e. that
$$\sum_\b G_\b|_M \e_\b = \sum_\b G_\b \rho (M) \e_\b = \sum_\c \sum_\b c^{i,M}_{\b,\c} \e_\c$$
and we get

\begin{align*}
  G_i|_M &= \sum_{\b \in D_i} G^{(i)}_\b|_M \e_\b 
          = \sum_{\b \in D_i} \FF(\Theta([i, \b]))|_M \e_\b \\
         &= \sum_{\b \in D_i} \FF(\eta(M)\Theta([i, \b])) \e_\b  & \text{by Lemma \ref{lem:oldforms:getOldForms}}\\
         &= \sum_{\b \in D_i} \FF(\Theta(\eta(M)[i, \b])) \e_\b  & \text{$\Theta$ is a hom. of reps}\\
         &= \sum_{\b \in D_i} \FF(\Theta(\rho_i(M)[i,\b])) \e_\b  & \text{def. of $\eta$}\\
         &= \sum_{\b \in D_i} \FF(\Theta( 
                 \sum_\c c^{i,M}_{\b,\c} [i,\c])) \e_\b \\
         &= \sum_{\b \in D_i} \sum_{\c \in D_i} c^{i,M}_{\b,\c}
                        \FF(\Theta([i,\c])) \e_\b \\
         &= \sum_{\b \in D_i} \sum_{\c \in D_i} c^{i,M}_{\b,\c}
                        G^{(i)}_\c \e_\b & \text{by \eqref{eq:defnG}} \\
         &= \sum_{\c \in D_i} G^{(i)}_\c \sum_{\b \in D_i} c^{i,M}_{\c,\b} \e_\b \\
         &= \sum_{\c \in D_i} G^{(i)}_\c \rho_i(M) \e_\c \\
         &= \rho_i(M) \sum_{\c \in D_i} G^{(i)}_\c \e_\c \\
         &= \rho_i(M) G_i
\end{align*}

\noindent Last but not least we claim that
$$F = \sum_{i=1}^n G_i\!\!\uparrow^{\text{func}}_{H_i} := G$$
We have
\begin{align*}
  G &= \sum_{i=1}^n \left(\sum_{\b \in D_i} G_{[i, \b]} \e_\b\right)\!\big\uparrow^{\text{func}}_{H_i} \\
    &= \sum_{i=1}^n \sum_{\b \in D_i} G_{[i, \b]} \uparrow_{H_i}\!\!(\e_\b) \\
    &= \sum_{i=1}^n \sum_{\b \in D_i} G_{[i, \b]} \sum_{\gamma \in \b} \e_\gamma \\
    &= \sum_{\gamma \in D} \sum_{i=1}^n \sum_{ \substack{\{\b \in D_i : \\ \gamma \in \b\}}} 
               ~\underbrace{G_{[i, \b]}}_{= \FF(\Theta([i, \b]))} \e_\gamma\\
    &= \sum_{\gamma \in D} \FF \circ \Theta 
		              \underbrace{ 
                    \left( \sum_{i=1}^n \sum_{ \substack{\{\b \in D_i :\\ \gamma \in \b\}}} \e_{[i, \b]} \right)
									}_{
									  = \downarrow(\e_\gamma) ~~ (\text{see \eqref{eq:downbasis}})
									} 
									\e_\gamma \\
    &= \sum_{\gamma \in D} \FF \circ \Theta\, \circ \downarrow (\e_\gamma) \e_\gamma \\
    &= \sum_{\gamma \in D} \FF \circ \iota \, \circ \downarrow (\e_\gamma) \e_\gamma & \text{as $\Theta$ is a continuation of $\iota$}\\
    &= \sum_{\gamma \in D} \FF \, \circ \downarrow(\e_\gamma) \e_\gamma\\
    &= \sum_{\gamma \in D} F_\gamma \e_\gamma & \text{by \eqref{eq:oldforms:Fvgamma}}\\
    &= F
\end{align*}

\noindent "$\Leftarrow$": Assume $F = A^{(1)}\!\uparrow_{H_1} + ... + A^{(n)}\!\uparrow_{H_n}$, then
\begin{equation}
  \label{eq:oldforms:FisOldForm}
  F_\gamma = \sum_{\substack{i \in \{1, ..., n\} \\ \gamma \in H_i^\bot}} A^{(i)}_{\gamma + H_i}
\end{equation}
We define a $\C$-linear map $B : X \to M_k(\Gamma(N))$ as
$$B([i, \b]) := A^{(i)}_\b, ~~\b \in D_i$$
and note that by \eqref{eq:oldforms:FisOldForm}, the diagram
\[ 
\xymatrix{ \C[D]\ar[rr]^{\F}  \ar[rd]_{v}     & & M_k(\Gamma(N)) \\
               & X \ar[ru]_{B} }
\]
commutes, i.e. $B(\downarrow\!\!(x)) = \F(x)$. In the language of "$\Rightarrow$", $B$ is $\FF$ 
and $\F$ factors through whole $X$, not only through $Y$. If $x \in \ker(\downarrow)$, then
$$0 = B(0) = B(\downarrow\!\!(x)) = \F(x)$$
and hence, $\ker(\downarrow) \subset \ker(\F)$.
\end{proof}

\section{An algorithm for splitting cusp forms into new and oldspace}
\label{sec:algorithm}
Let $\D = (D, Q)$ be a discriminant form. 
On $\C[D]$ there is a canonical 
scalar product, namely the sesquilinear continuation of
$$\ideal{\e_\gamma, \e_\delta} = \one_{\gamma = \delta}$$
i.e., the canonical basis $(\e_\gamma)_{\gamma \in D}$ forms an orthonormal basis.
Similarly, for isotropic subgroups $H_1, ..., H_n$ of $D$, on $X := \C[D_1] \oplus ... \oplus \C[D_n]$ 
-- where $D_i = H_i^\bot/H_i$ -- we can define a scalar product by putting the single ones together, 
i.e. if we identify $X$ with $\C[\sqcup_{i=1,...,n} D_i]$, and denote the canonical 
basis just by $[i, \a]$ (instead of $\e_{[i,\a]}$), then this basis forms an orthonormal basis.
We call these scalar products $\ideal{\cdot, \cdot}_{\C[D]}$, respectively $\ideal{\cdot, \cdot}_X$.
Similarly, we can put the Petterson products on $\vvcf{D_i}$ together
in order to obtain a scalar product on $\vvcf{X} \cong \vvcf{\D_1} \orthplus ... \orthplus \vvcf{\D_n}$.
We verify:
\begin{lem}
  \label{lem:up-down-adjoint}
  Let $\D=(D, Q)$ be a discriminant form of even signature and $H_1, ..., H_n$ isotropic subgroups of $D$.
	Put $D_i := H_i^\bot/H_i$ and $X=\C[\sqcup_{i=1,...,n} D_i]$ as above. Then
	\begin{enumerate}[(i)]
		\item
	    \label{lem:up-down-adjoint:algebraic}
		  $\bilform{\uparrow\!\!(\zeta), w}_{\C[D]} = \bilform{\zeta, \downarrow\!\!(w)}_X
			~,~\zeta \in X, w \in \C[D]$.
		\item
      \label{lem:up-down-adjoint:func}
		  $\bilform{\uparrow^\text{func}\!\!(G), F}_{\vvcf{\D}} = \bilform{G, \downarrow^\text{func}\!\!(F)}_{\vvcf{X}}
			~,~G \in \vvcf{X}, F \in \vvcf{\D}$.
	\end{enumerate}
  i.e. up arrow and down arrow are mutually adjoint.
\end{lem}
\begin{proof}
  \eqref{lem:up-down-adjoint:algebraic}:
  Since everything is sesquilinear, we only need to verify this for the basis vectors
  $\zeta = [i,\a]$ and $w = \e_\gamma$.
  \begin{align*}
    \ideal{\uparrow\!\!(\zeta), w}_{\C[D]} 
      &= \ideal{\sum_{\mu \in \a} \e_\mu , \e_\gamma}_{\C[D]} \\
      &= \begin{cases}
           1 & \text{if $\gamma \in \a$} \\
           0 & \text{otherwise}
         \end{cases}
  \end{align*}
  \noindent and
  \begin{align*}
    \ideal{\zeta, \downarrow\!\!(w)}_X 
      &= \ideal{[i,\a] , \sum_{\gamma \in H_j^\bot} [j, \gamma + H_j]}_X \\
      &= \begin{cases}
           1 & \text{if there is a $j$ with $[i,\a] = [j, \gamma + H_j]$} \\
           0 & \text{otherwise}
         \end{cases} \\
      &= \begin{cases}
           1 & \text{if $\a = \gamma + H_j$} \\
           0 & \text{otherwise}
         \end{cases}
  \end{align*}
	
  \noindent \eqref{lem:up-down-adjoint:func}: This is a straightforward computation 
	analogously to the one in \eqref{lem:up-down-adjoint:algebraic}.
\end{proof}

\noindent Recall that, in order not to confuse the up/down maps on vector valued modular forms and their purely algebraic parts,
we give them different names: The up/down maps on vector valued modular forms are denoted by
$\uparrow^{\text{func}}, \downarrow^\text{func}$
and their algebraic parts are named $\uparrow, \downarrow$.
We summarize:
\begin{thm}
  \label{thm:oldforms-summarize}
  Let $\D = (D, Q)$ be a discriminant form, $H_1, ..., H_n$ isotropic subgroups.
  Let $F \in \vvcf{\D}$ and $\F$ its associated evaluation map as in Def. \ref{dfn:evaluation-map}.
  Then
  \begin{align*}
    F ~\text{is an oldform w.r.t. $H_1, ..., H_n$} &\iff F \in \im(\uparrow^\text{func})      \\
                   &\iff \ker(\downarrow) \subseteq \ker(\F) \\
    \text{ }\\
    F ~\text{is a newform w.r.t. $H_1, ..., H_n$}  &\iff F \in \im(\uparrow^\text{func})^\bot \\
                   & \iff \im(\uparrow) \subseteq \ker(\F) \\
       & \iff \forall i=1,...,n ~~ \forall \gamma \in H_i^\bot \sum_{h \in H_i} F_{\gamma + h} = 0
  \end{align*}
\end{thm}

\begin{proof}
  The first line was shown in Thm. \ref{thm:oldforms:detectionThm}.
  On the second line:
  By definition,
  $$\vvnf{\D}{H_1, ..., H_n} = \left(\vvcf{\D_1}\!\uparrow^{\text{func}}_{H_1} + ... + \vvcf{\D_n}\!\uparrow^{\text{func}}_{H_n}\right)^\bot$$
  Generally speaking, for every pair of suspaces $A,B$ of a vector space with bilinear form, 
  $(A+B)^\bot = A^\bot \cap B^\bot$ so
  \begin{equation}
    \label{eq:newforms-star}
    \vvnf{\D}{H_1, ..., H_n} = \vvnf{\D}{H_1} \cap ... \cap \vvnf{\D}{H_n}
  \end{equation}
  "$\Rightarrow$":
  \begin{align*}
    F ~\text{is new} & \Rightarrow F \in \vvnf{\D}{H_1, ..., H_n} 
                    \stackrel{\eqref{eq:newforms-star}}{\Rightarrow}
                                F \in \vvnf{\D}{H_i} ~~\forall i \\
                    & \Rightarrow \ideal{F\!\!\downarrow^{\text{func}}_{H_i}, g}  
										\stackrel{\text{Lemma} \ref{lem:up-down-adjoint}\eqref{lem:up-down-adjoint:func}}{=} 
										  \ideal{F, g\!\uparrow^{\text{func}}_{H_i}} = 0 ~~\forall i ~\forall g \in \vvcf{\D_i} \\
                    & \Rightarrow F\!\!\downarrow^{\text{func}}_{H_i} \in \vvcf{\D_i}^\bot = \{0\} ~~ \forall i\\
                    & \Rightarrow F \in \bigcap_{i=1}^n \ker(\downarrow_{H_i}^\text{func}) = \ker((\downarrow_{H_1}^\text{func}, ..., \downarrow_{H_n}^\text{func})) = \ker(\downarrow^\text{func})
  \end{align*}
  "$\Leftarrow$":
  \begin{align*}
    F \in \ker(\downarrow^\text{func}) &= \bigcap_{i=1}^n \ker(\downarrow_{H_i}^\text{func}) \\
                    &\Rightarrow \ideal{F, g\!\uparrow^{\text{func}}_{H_i}}  
										\stackrel{\text{Lemma} \ref{lem:up-down-adjoint}\eqref{lem:up-down-adjoint:func}}{=} 
										  \ideal{F\!\!\downarrow^{\text{func}}_{H_i}, g} = \ideal{0, g} = 0 \\
                    &\Rightarrow F \in \bigcap_{i=1}^n \left(\vvcf{\D_i}\!\uparrow^{\text{func}}_{H_i}\right)^\bot 
										  = \vvnf{\D}{H_1, ..., H_n}
  \end{align*}
  Now
  \begin{align*}
    F \in \ker(\downarrow_{H_i}) & \iff 0 = F\downarrow_{H_i} = \sum_{\b \in D_i} \left(\sum_{\gamma \in \b}  F_\gamma \right) [i, \b] \\
                & \iff \sum_{\gamma \in \b} F_\gamma = 0 ~~\forall \b \in H_i^\bot / H_i
  \end{align*}
\end{proof}

\noindent This gives an algorithm for concretely computing the decomposition
  $$ \vvcf{\D} = \vvofwo{\D} \orthplus \vvnfwo{\D}$$
using a computer algebra system.
First we compute the set of all isotropic subgroups we are interested in, say $H_1, ..., H_k$.
This is possible as $D$ is a finite set!
We compute a basis of $\vvmf{\D}$. We can use, for example, the algorithm by M. Raum \cite{raum}.
As a result we get the first parts of the Fourier expansions of a basis $F_1, ..., F_m$ 
of vector valued modular forms up to a certain number nowadays known as the sturm bound, i.e.
we know $a_{n, \gamma}(F_i)$ for all $i=1,...,m$ and $n=0,...,S$ where $S$ is a fixed natural.
After doing this, we set up the system for determining all
$\lambda_1, ..., \lambda_m \in \C$ with the property that $\sum_{i=1}^m \lambda_i F_i \in \vvnf{\D}{H_1, ..., H_k}$.
This is easy: once we have truncated to the sturm bound, this is a finite dimensional linear 
system of equations due to Thm \ref{thm:oldforms-summarize}, namely we have to compute those $\lambda_i$ with
$$\sum_{i=1}^{n} \sum_{\gamma \in \b} \lambda_i a_{n, \gamma}(F_i) = 0 ~~~ n=0,1,...,S$$
where we let $\b$ run through all the classes in each $H_l^\bot / H_l$ for $l=1,...,k$.

Analogously, we can compute all oldforms w.r.t. $H_1, ..., H_k$ by first computing the 
kernel of $\downarrow$ (finite dimensional linear system!) and then computing in the same way as above all
$\lambda_1, ..., \lambda_m$ with $\ker(\downarrow) \subset \ker(\F)$ where $F = \sum_i \lambda_i F_i$.
We can truncate this to all Fourier coefficients $n=0,1,...,S$ so this again becomes a finite dimensional linear system.

\section{Preparations}
\label{sec:preparation}
Having proved a neat criterion for detecting oldforms, in this section we do some preparations for the proof of the main theorem.
We want to show that all forms are oldforms. Indeed, it suffices to show that the algebraic part $\uparrow$ is surjective. The surjectivity of $\uparrow^{\text{func}}$ then follows:

\begin{lem}
  \label{lem:all-old}
  Let $\D = (D, Q)$ be a discriminant form. Assume that $\uparrow$ 
	(involving all isotropic subgroups) is surjective. 
  Then, so is $\uparrow^{\text{func}}$. In other words: every vector valued modular form for 
  $\D$ is an oldform.
\end{lem}
\begin{proof}
  By Lemma \ref{lem:up-down-adjoint}, $\uparrow$ and $\downarrow$ are mutually adjoint to each other.
  This implies $\ker(\downarrow) = \im(\uparrow)^\bot = \C[D]^\bot = \{0\}$.
  Hence, the condition in Theorem \ref{thm:oldforms:detectionThm} becomes trivial.
\end{proof}

\begin{dfn}
  Let $\D = (D, Q)$ be a discriminant form and $n \in \N$.
	Every sequence consisting of $n+1$ isotropic subgroups $H_0, ..., H_n$ such that
  \begin{enumerate}[(a)]
    \item $H_i \bot H_j$ for all $i \neq j$
      \label{dfn:nicely-orthogonal:all-orthogonal}
    \item $H_0 + (H_i \setminus \{0\}) \subseteq \bigcup_{k=1}^{m} H_k$ for all $i=1, ..., n$
      \label{dfn:nicely-orthogonal:closed}
    \item All the $H_i$ are cyclic and of the same size $n$, 
          i.e. $H_i = \langle \gamma_i \rangle$ for some $\gamma_i \in D$ 
          and $|H_i| = n$ for $i=0, ..., n$.
      \label{dfn:nicely-orthogonal:cyclicity}
    \item All the pairs $\gamma_i, \gamma_j$ for $i,j \in \{0, ..., n\}$ with $i \neq j$ 
          are 'weakly $\Z$-linearly independent'
          meaning that whenever there are $a, b \in \Z$ such that $a\gamma_i = b\gamma_j$ then
          $a\gamma_i=b\gamma_j=0$.
      \label{dfn:nicely-orthogonal:linear-independence}
  \end{enumerate}
  is called a sequence of $n+1$ nicely orthogonal isotropic subgroups.
  We say that this is a sequence of $n+1$ nicely orthogonal isotropic 
  subgroups for some $\gamma \in D$ iff.
  it is a sequence of $n+1$ nicely orthogonal isotropic subgroups and
  $\gamma \in H_i^\bot$ for all $i=0,1,...,n$.
\end{dfn}

\begin{lem}
  \label{lem:egammaimage}
  Let $\D = (D, Q)$ be a discriminant form and $\gamma \in D$.
  Let $H_0, ..., H_n$ be a sequence of $n+1$ nicely orthogonal 
  isotropic subgroups for $\gamma$, then
  $\e_\gamma \in \im(\uparrow)$. In fact,
  $\e_\gamma \in \im(\uparrow|_{\C[\sqcup_{i=0,...,n} D_i]})$ 
  where $D_i = H_i^\bot / H_i$.
\end{lem}
\begin{proof}
  Let $\D_i = H_i^\bot / H_i$.
  Put 
    $$M := \bigdotcup_{i=1, ..., n} \gamma + H_i \setminus \{ \gamma \}$$
  The union is indeed disjoint:
  Let $\mu \in \gamma + H_i \setminus \{\gamma\} \cap \gamma + H_j \setminus \{\gamma\}$ for $i \neq j$.
  Then there are $h_i \in H_i, h_j \in H_j$ such that
    $$\mu = \gamma + h_i = \gamma + h_j$$
  and $h_i \neq 0, h_j \neq 0$ as $\mu \neq \gamma$.
  Hence, $h_i = \mu - \gamma = h_j$. The $H_i$ are cyclic by assumption, so there are $a, b \in \Z$ with
  $h_i = a\gamma_i, h_j = b\gamma_j$. We obtain $a\gamma_i = h_i = h_j = b\gamma_j$ so
  $h_i = a\gamma_i = 0 = b\gamma_j = h_j$ by assumption 
  \eqref{dfn:nicely-orthogonal:linear-independence}, a contradiction.

  \noindent 
  We claim that there are precisely $n-1$ cosets $\a_1, ..., \a_{n-1}$ in $\D_0$ such that
  \begin{equation}
    \label{eq:Mset}
    M = \bigdotcup_{j=1, ..., n-1} \a_j
  \end{equation}
  In order to show this we first show that $M$ is $H_0$ invariant, i.e. for every $\mu$ in $M$,
  $\mu + h \in M$ for all $h \in H_0$:
  Let $\mu = \gamma + h_j$ with $h_j \in H_j$ for some $j\in \{1, ..., n\}$ and,
  as we only take $\gamma + H_j \setminus \{\gamma\}$, $h_j \neq 0$.
  Let $h \in H_0$ be arbitrary. By assumption \eqref{dfn:nicely-orthogonal:closed}, 
  $h + h_j = h_v \in H_v$ for some $v \in \{1, ..., n\}$. Hence,
  $$ \mu + h = \gamma + h + h_j = \gamma + h_v \in \bigcup_{i=1, ..., n} \gamma + H_i$$
  In order to see that $\mu + h \in M$ we therefore only need to show $\mu + h \neq \gamma$.
  Assume $\gamma + h_j + h = \mu + h = \gamma$ then $h + h_j = 0$, thus $h = -h_j$.
  As $h_j \neq 0$, also $h \neq 0$. By the cyclicity of the $H_i$, there are $a, b \in \Z$ such that
  $h = a\gamma_0$ and $h_j = b\gamma_j$. Consequently,
  $a\gamma_0 = h = -h_j = -b\gamma_j$. By Assumption \eqref{dfn:nicely-orthogonal:linear-independence}
  $a\gamma_0=-b\gamma_j=0$ i.e. $h_j = b\gamma_j = -(-b\gamma_j) = 0$ follows. Contradiction.
  In total: $\mu+h \neq \gamma$ and $\mu + h \in M$ and the $H_0$-invariance of $M$ is shown.
  Put
    $$S := \bigcup_{\mu \in M} \mu + H_0$$
  then clearly $M \subseteq S$ but we also have $S \subseteq M$ by the above:
  If $\mu \in M$ and $h \in H_0$ then also $\mu + h \in M$, hence,
  for every $\mu \in M$, $\mu + H_0 \subseteq M$ and therefore,
    $$S = \bigcup_{\mu \in M} \mu + H_0 = M$$
  Choose representatives $\lambda_1, ..., \lambda_A$ for the equivalence relation
  $$x \sim y \iff ~ \exists ~ h_0 \in H_0 ~~ x = y + h_0$$
  on $M$ then $M = \cup_{i=j, ..., A} \lambda_j + H_0$. We measure the size of both sides:
  Firstly, $|M| = n \cdot |\gamma + H_i \setminus \gamma| = n (n-1)$ (as $|H_i| = n$ for all $i$) and therefore
  $n(n-1) = |M| = |S| = A\cdot n$, so $A = n-1$. If we put $\a_j = \lambda_j + H_0$, we have shown
  \eqref{eq:Mset}.
  Now we construct a concrete preimage for $\e_\gamma$:
  We put 
   $$\zeta := -\frac{1}{n} \sum_{j=1, ..., n-1} [0, \a_j] + \frac{1}{n} \sum_{i=1, ..., n} [i, \gamma + H_i]$$
  As $\gamma \in H_i^\bot$ for all $i$, $\gamma + H_i$ is a class in $\D_i$,
  this is a well defined element of $\C[\sqcup_{i=0,...,n} D_i]$ which is a subset (and a subspace) 
  of $\C[\sqcup_{H} D_H]$ (the union runs over all isotropic subgroups of $D$). 
  We compute
  \begin{align*}
    \uparrow\!\!(\zeta) &= 
        -\frac{1}{n} \sum_{j=1, ..., n-1} \uparrow\!\![0, \a_j] + \frac{1}{n} \sum_{i=1, ..., n} \uparrow\!\![i, \gamma + H_i] \\
             &=
        -\frac{1}{n} \sum_{j=1, ..., n-1} \sum_{\mu \in \a_j} \e_\mu + 
         \frac{1}{n} \sum_{i=1, ..., n} \left( \sum_{\mu \in \gamma + H_i \setminus \{ \gamma\} } \e_\mu + \e_\gamma \right) \\
             &=
         \frac{1}{n} \left(
             -\sum_{\mu \in \cup_{j=1}^{n-1} \a_j} \e_\mu + 
              \sum_{\mu \in \cup_{i=1,...,n} \gamma + H_i \setminus \{\gamma\}} \e_\mu \right)
               + n\frac{1}{n} \e_\gamma \\
             &=
         \frac{1}{n} \left(
             -\sum_{\mu \in S} \e_\mu + 
              \sum_{\mu \in M} \e_\mu \right)
               + \e_\gamma \\
             &= 0+\e_\gamma = \e_\gamma & \text{by \eqref{eq:Mset}}
  \end{align*}
  Note that we have used the disjointness of the unions in the definitions of $S$ and $M$
  to transform the sums into 'union' symbols.
\end{proof}

We see that we need a mechanism that allows us to construct nicely 
orthogonal subgroups for all elements $\gamma \in D$.
The next Lemma provides us with such a method:

\begin{lem}
  \label{lem:existence-nicely-orth-gamma-general}
  Let $\D = (D, Q)$ be a discriminant form and $\gamma \in D$. 
  Let $\gamma^\bot = \{ \mu \in D : (\mu, \gamma) = 0 + \Z \}$.
  Assume there are two isotropic vectors $\delta, \mu$ in $\gamma^\bot$
  a prime $p$ (not necessarily odd!) and a natural $e \in \N$ such that
  \begin{enumerate}
    \item
    Whenever $a,b \in \Z$ are such that $a\delta + b\mu = 0$ then
    $a \equiv b \equiv 0 \mod p^e$.
    \label{lem:existence-nicely-orth-gamma-general:lin-ind}
    \item
    $\delta \bot \mu$.
    \label{lem:existence-nicely-orth-gamma-general:orth}
  \end{enumerate}
  Then there exists a sequence of $p+1$ nicely orthogonal isotropic subgroups for $\gamma$.
  Consequently, $\e_\gamma \in \im(\uparrow)$.
\end{lem}

\begin{proof}
  When $x,y \in \Z$ or $x,y \in \Z_n$, we write $[x,y]$ in place for $x \delta + y\mu$.
  Let $q := p^e$.
  We define
  $$h_{-1} := p^{e-1} [0,1] , ~~~ h_j := p^{e-1} [1,j] ~\text{for $j=0,1,...,p-1$}$$
  and
  $$H_j := \langle h_j \rangle \text{for $j=-1,0,1,...,p-1$}$$

  \noindent These are subgroups of order $p$:
  for if, say for $j\geq 0$, $v \in \Z$ with $v h_j = 0$ then
  $vp^{e-1}\delta + vp^{e-1}j\mu = 0$. By assumption \eqref{lem:existence-nicely-orth-gamma-general:lin-ind}, 
  $vp^{e-1} \equiv 0 \mod p^e$ but this holds iff. 
  $v \equiv 0 \mod p$. Analogously we proceed with $h_{-1}$.
  Hence, $H_j = \{ 0, h_j, 2h_j, ..., (p-1)h_j\}$.
  We verify the properties of nicely orthogonal subgroups:

  \noindent \eqref{dfn:nicely-orthogonal:all-orthogonal}:
  First let $i,j \geq 0$ then
  $$(h_i, h_j) = (\delta + i\mu, \delta + j\mu) = (\delta, \delta) + ij(\mu, \mu) = 0 + 0 = 0$$
  as $\delta \bot \mu$ and $\delta, \mu$ are isotropic.
  Analogously we verify this for $(h_{-1}, h_j)$
  
	\noindent \eqref{dfn:nicely-orthogonal:closed}:
  Let $x_{-1} \in H_{-1}$ and $x_j \in H_j \setminus \{0\}$ for some $j$.
  By definition, the $H_i$ are cyclic, so there are $\alpha, \beta \in \{0,1,...,p-1\}$ such that
  $x_{-1} = \alpha h_{-1} = [0, \alpha]$ and $x_j = \beta h_j = [\beta, \beta j]$.
  As $x_j \neq 0, \beta \neq 0$ so we can invert $\beta$ in $\Z_p$ and get
  Now 
    $$x_{-1} + x_j = \beta [1, \frac{\alpha + j\beta}{\beta}]$$
  so this is an element in $H_k$ where $k \equiv \beta^{-1} (\alpha + j\beta) \mod p$.
	
  \noindent \eqref{dfn:nicely-orthogonal:cyclicity}: See above.
	
  \noindent \eqref{dfn:nicely-orthogonal:linear-independence}:
  Let $a,b \in \Z$ be such that $ah_i + bh_j = 0$. First assume $i,j \geq 0$. Then
    $$0 = ah_i + bh_j = [a, ai] + [b, bj] = [a+b, ai + bj]$$
  By assumption \eqref{lem:existence-nicely-orth-gamma-general:lin-ind}, 
  it follows that $a+b \equiv ai + bj \equiv 0 \mod p$. Rephrased in matrix 
  language this means
    $$\mat{1 & 1 \\ i & j} \mat{a \\ b} \equiv \mat{0 \\ 0} \mod p$$
  As $\tmat{1 & 1 \\ i & j}$ is invertible over $\Z_p$ (because $i \neq j$), 
  this means $a \equiv b \equiv 0 \mod p$.
  Now assume $i=-1$ and $j \geq 0$ then
    $$0 = ah_{-1} + bh_j = [0, a] + [b, bj] = [b, a + bj]$$
  By assumption \eqref{lem:existence-nicely-orth-gamma-general:lin-ind}, 
  this implies $b \equiv 0 \mod p$ and hence, $0 = [0, a]$ so again, 
  by assumption \eqref{lem:existence-nicely-orth-gamma-general:lin-ind}, 
  $a \equiv 0 \mod p$.
  
  \noindent Hence, $H_{-1}, H_0, H_1, ..., H_{p-1}$ is a sequence of $p+1$ 
  nicely orthogonal isotropic subgroups. It is a sequence for $\gamma$
  because the $h_i$ are in the span of $\delta, \mu$ and they lie -- by assumption -- 
  in $\gamma^\bot$, so $h_j \bot \gamma$ for all $j=-1,0,1,...,p-1$ or, 
  as the $H_j$ are cyclic, $\gamma \in H_j^\bot$.

\end{proof}

\begin{dfn}
  Let $R$ be a commutative ring and $M$ a freely, finitely generated $R$ module of rank $r$.
  Choose a basis $m_1, ..., m_r$. An element
  $m = \sum_{i=1}^r a_i m_i$ is called primitive w.r.t. this basis iff. there exists an $i$ such that $a_i \in R^*$.
\end{dfn}

\noindent Let  $R=\Z_{p^e}$ for some prime $p$ and $e \in \N$. Elements of $R$ are equivalence classes $[a]$ 
but the assertion "$p|a$" does not depend on the chosen representative. 
For these rings we get
$$\gamma ~\text{primitive} \iff \exists i ~~ p \ndivides a_i$$
in particular, if $m \neq 0$ is not primitive (w.r.t. to some fixed basis) then
$p|a_i$ for all $i$ and hence we can pull out all $p$-powers and end up 
at $m = p^r m'$ for some primitive $m' \in M$.
For $R=\Zp$, the $p$-adic integers, we obtain precisely the same results.

\begin{rmk}
  Let $R$ be an integral domain and $M$ a freely, finitely generated $R$-module.
	Let $(\cdot, \cdot)$ be a bilinear form on $M$.
  For a basis  $m=\{m_1, ..., m_r\}$ of $M$ we consider the Gram matrix
    $$G_m = ((m_i, m_j))_{i,j=1,...,n}$$
  then 
  \begin{align*}
    (\cdot, \cdot) ~\text{non-degenerate} &\iff \det(G_m) \neq 0 ~\text{for all bases $m$}\\
     &\iff \det(G_{m_0}) \neq 0 ~\text{for one fixed basis $m_0$}
  \end{align*}
  and
  \begin{align*}
    (\cdot, \cdot) ~\text{non-degenerate} &\iff \det(G_m) \in \R^\times (\iff G_m \in \GL_r(R)) ~\text{for all bases $m$}\\
     &\iff \det(G_{m_0}) \in R^\times ~\text{for one fixed basis $m_0$}
  \end{align*}
\end{rmk}

\begin{notation}
\label{not:padicStuff}
We set up some terminology which we will use from now on.
Let $p$ be a fixed prime (not necessarily odd) and $e \in \N$.
For making the interaction between $\Z, \Z_{p^e}$ and $\Zp$ rigorous we need
to name each of the several maps that are floating around in between them.
Firstly, there is 
  $$\overline{\cdot} : \Z \to \Z_{p^e}, ~~ x \mapsto \overline{x} := x + p^e\Z$$
secondly, there is the imbedding
  $$\iota : \Z \to \Zp, ~~~ x \mapsto \iota(x)$$
but in a clear abuse of notation, for the sake of readability 
we will drop $\iota$ as often as possible.
Occasionally, we will non the less remark that this map is involved.
We recall one more map:
Every element $\alpha \in \Zp$ can be written uniquely as an infinite power series
$\alpha = \alpha_0 + \alpha_1p + \alpha_2 p^2 + ...$ with $\alpha_i \in \{0, 1, ..., p-1\}$.
We define another map
$$R_{p^e}^\Z : \Zp \to \Z, ~~ R_{p^e}^\Z(\alpha) := \alpha_0 + \alpha_1p + ... + \alpha_{p^{e-1}} p^{e-1}$$
We also define
$$R_{p^e} : \Zp \to \Z_{p^e}, ~~ R_{p^e} := \overline{\cdot} \circ R_{p^e}^\Z$$
to be the so-called reduction of $\alpha$ modulo $p^e$.
The maps $\overline{\cdot}, \iota$ and $R_{p^e}$ are ring homomorphisms 
(careful: $R_{p^e}^\Z$ is not!). These maps
are also defined on vectors or matrices over their respective domains 
by applying them component wise.
It is important to note that
\begin{equation}
\label{eq:RpeCircTildeIsOverline}
  R_{p^e} \circ \iota = \overline{\cdot}
\end{equation}
on elements in, and vectors and matrices over $\Z$.
We also let $\nu_p$ denote the $p$-adic valuation 
on the $p$-adic integers $\Zp$ throughout: every $\alpha \in \Zp$ can be written
uniquely as $\alpha = \epsilon p^r$ for some $r \in \N \cup \{0\}$ 
and $\epsilon \in \Zp^\times$. Then, $\nu_p(\alpha) := r$.
\end{notation}

\begin{lem}
  \label{lem:primitive-high-norm-splits}
  Let $M$ be a freely, finitely generated $\Zp$ module or rank $r \geq 2$ for a 
  (not necessarily odd!) prime $p$. Suppose
  $\pform{\cdot, \cdot}$ is a symmetric \textbf{unimodular} bilinear form on $M$.
  Let $\tilde{\gamma} \in M$ such that $\nu_p(\pform{\tilde{\gamma}, \tilde{\gamma}}) > 0$
  and $\tilde{\gamma}$ is primitive.
  Then there exists an element $\tilde{\delta} \in M$ such that
  \begin{enumerate}[(a)]
   \item 
   \label{lem:primitive-high-norm-splits:independent}
   $\tilde{\gamma}$ and $\tilde{\delta}$ are $\Zp$--linearly independent
   \item 
   \label{lem:primitive-high-norm-splits:split}
   The submodule $U := \Zp \tilde{\gamma} \oplus \Zp \tilde{\delta}$ 
   can be split off orthogonally, i.e.
   $M = U \orthplus U^\bot$ and $U^\bot$ is freely, 
   finitely generated of rank $r-2$.
  \end{enumerate}
\end{lem}

\begin{proof}
Let $\tilde{G} \in \GL_n(\Zp)$ denote the (invertible) Gram matrix of $\pform{\cdot, \cdot}$ 
with respect to any fixed basis of $M$. We view vectors as column vectors and their entries are the coordinates $\Zp$ w.r.t. this basis.
$\tilde{\gamma}$ is primitive, consequently there exists a coordinate $\tilde{\gamma}_i \in \Zp^\times$.
As $\tilde{G}$ is invertible over $\Zp$, there is a vector $\tilde{\delta} \in M$ such that 
$\tilde{G} \tilde{\delta} = e_i$.
($e_i$ is the column vector having $0$ at every position except at $i$ and $1$ at $i$).
Hence, $\pform{\tilde{\gamma}, \tilde{\delta}} = \tilde{\gamma}^T \cdot \tilde{G} \cdot \tilde{\delta} = \tilde{\gamma}^T \cdot e_i = \tilde{\gamma}_i \in \Zp^\times$
(here, $T$ means 'transpose').
If we rescale $\tilde{\delta}$ by $\tilde{\gamma}_i^{-1}$ then we get $\pform{\tilde{\gamma}, \tilde{\delta}} = 1$.
This already suffices to see that $\tilde{\gamma}, \tilde{\delta}$ are $\Zp$-linearly independent:

Let $\pform{\tilde{\gamma}, \tilde{\gamma}} = p^w a$ and 
$\pform{\tilde{\delta}, \tilde{\delta}} = p^s b$ with $a,b \in \Zp^\times$.
Suppose $x,y \in \Zp$ have the property that $x\tilde{\gamma} + y\tilde{\delta} = 0$.
Pairing this expression with $\tilde{\gamma}$ yields
$$0 = \pform{0, \tilde{\gamma}} 
    = \pform{x\tilde{\gamma} + y\tilde{\delta}, \tilde{\gamma}}
    = x \pform{\tilde{\gamma}, \tilde{\gamma}} + y\pform{\tilde{\gamma}, \tilde{\delta}}
    = x p^w a + y$$
so $y = -xa p^w$.
Pairing the expression with $\tilde{\delta}$ yields
$$0 = \pform{0, \tilde{\delta}} 
    = \pform{x\tilde{\gamma} + y\tilde{\delta}, \tilde{\delta}}
    = x \pform{\tilde{\gamma}, \tilde{\delta}} + y\pform{\tilde{\delta}, \tilde{\delta}}
    = x + p^s by$$
In matrix notation, this means
$$\mat{p^wa & 1 \\ 1 & p^s b} \mat{x\\y} = \mat{0\\0}$$
but this matrix is invertible over $\Zp$ as its determinant is 
$p^{w+s} ab - 1 \in \Zp^\times + p\Zp \subset \Zp^\times$ 
(we use $w > 0$, i.e. $p^{w+s}ab \in p\Zp$ here!).
Hence, $x=y=0$ follows and
the submodule $U = \Zp \tilde{\gamma} + \Zp \tilde{\delta}$ is in fact
$U= \Zp \tilde{\gamma} \oplus \Zp \tilde{\delta}$, a free module of rank $2$. Its Gram matrix is
$$H = \mat{p^w a & 1 \\ 1 & p^s b}$$
in particular, as we have seen above, $\det(H)$ is a unit in $\Zp$. Consequently,
$U$ is unimodular and therefore it can be split off orthogonally (see \cite{kneser}, Satz 1.6 on p.2), i.e.
$M = U \orthplus U^\bot$. As $\Zp$ is a principal ideal domain and $U^\bot$ is a submodule of 
the freely, finitely generated $\Zp$ module $M$, $U^\bot$ is free again 
(see \cite{jantzen-schwermer}, chapter VII, Satz 8.3 on p.172) and
$$r = \operatorname{rank}(M) = \operatorname{rank}(U) + \operatorname{rank}(U^\bot) = 2 + \operatorname{rank}(U^\bot)$$
so $\operatorname{rank}(U^\bot) = r-2$.
\end{proof}

\begin{lem}
\label{lem:finish:p=odd}
Let $p$ be an odd prime, $e \in \N$, put $q := p^e$ and let $D$ be a discriminant form with $D \cong (\Z_q)^n$ with
$$n \geq 
  \begin{cases} 
	  5 & \text{if $e=1$} \\
		2 & \text{if $e\geq 2$}
	\end{cases}
$$
then $D$ contains two isotropic, orthogonal, $\Z_p$--linearly independent vectors.
\end{lem}

\begin{proof}
Let $e=1$, i.e. $q=p$ for an odd prime $p$.
Let $\Gamma = \{\gamma_1, ..., \gamma_n\}$ be such that
$D = \Z_q \gamma_1 \oplus ... \oplus \Z_q \gamma_n$.
We let $H, G, \tilde{G}$ be as in Rmk. \ref{rmk:change-of-basis-Zp-vs-ZmodpZ}.
Choose a fixed $\epsilon \in \Zp^\times$ that is not a square
(in fact, one can choose $\epsilon \in \Z$ such that $(\epsilon,p)=1$ and
$\epsilon$ is not a square in any of the rings $\Z_{p^r}, r \in \N$), then
$$\Zp^\times/(\Zp^\times)^2 = \{(\Zp^\times)^2, \epsilon (\Zp^\times)^2 \}$$
(see \cite{cassels}, Cor. on p.40 or almost any other book on $p$-adic numbers).
We put
\begin{align*}
  \tilde{A} &:= \diag(1, -1, 1, -1, 1, ..., 1, 1) \in \GL_n(\Zp),\\
	\tilde{B} &:= \diag(1,-1,1,-1,1, ..., 1,\epsilon) \in \GL_n(\Zp)
\end{align*}
As $\det(\tilde{A}) = 1, \det(\tilde{B}) = \epsilon$, the determinants of
these forms exhaust $\Zp^\times / (\Zp^\times)^2$ completely. By 
Thm. \ref{thm:only-2-unimodular-forms-over-odd-p}
the bilinear form induced by $\tilde{G}$ is either isomorphic to 
the one induced by $\tilde{A}$ or to the one induced by $\tilde{B}$.
Hence, we get an $\tilde{S} \in \GL_n(\Zp)$ such that either
$\tilde{S}^T \tilde{G} \tilde{S} = \tilde{A}$ or 
$\tilde{S}^T \tilde{G} \tilde{S} = \tilde{B}$.
In any case, using Rmk. \ref{rmk:change-of-basis-Zp-vs-ZmodpZ}, we obtain a new basis
$D = \Z_p \delta_1 \orthplus ... \orthplus \Z_p \delta_n$ such that the Gram matrix w.r.t. 
this basis is given by
$p^{-1} R_{p}(\tilde{S}^T \tilde{G} \tilde{S}) + \Z$ which is either
$p^{-1} R_{p}(\tilde{A}) + \Z$ or $ = p^{-1} R_{p}(\tilde{B}) + \Z$.
In either case, the first part looks like $p^{-1} \diag(1, -1, 1, -1, ...)$
so, $\delta_1+\delta_2, \delta_3+\delta_4$ is a pair of orthogonal, 
isotropic, $\Z_p$--linearly independent vectors.

In the case that $e>1$, we use Thm. \ref{thm:jordan} to choose a 
Jordan decomposition, i.e.
a basis such that $D = \Z_q \gamma_1 \orthplus ... \orthplus \Z_q \gamma_n$.
Then $p^{e-1}\gamma_1, p^{e-1}\gamma_2$ are isotropic
($Q(p^{e-1}\gamma_i) = p^{2(e-1)}Q(\gamma_i) = p^{2(e-1)}\tfrac{*}{p^e} + \Z = 0 + \Z$ 
as $2(e-1) \geq e$ as $e \geq 2$) and as $\gamma_1, \gamma_2$ were $\Z_q$--linearly independent, 
$p^{e-1}\gamma_1, p^{e-1}\gamma_2$ are $\Z_p$--linearly independent.
As $\gamma_1, \gamma_2$ were orthogonal, $p^{e-1}\gamma_1, p^{e-1}\gamma_2$ are orthogonal.
\end{proof}

\begin{lem}
\label{lem:finish:p=2}
Let $e \in \N$, $q := 2^e$ and let $D$ be a discriminant form with $D \cong (\Z_q)^n$ with
$$n \geq 
  \begin{cases} 
	  7 & \text{if $e=1$ or $e=2$} \\
		3 & \text{if $e\geq 3$}
	\end{cases}
$$
then $D$ contains two isotropic, orthogonal $\Z_2$-- or $\Z_q$--linearly independent vectors.
\end{lem}
\begin{proof}
We take any fixed Jordan decomposition of $D$ (see Thm. \ref{thm:jordan}).
By basic algebra, the decomposition of an abelian finite group into powers of 
$\Z_{p^r}$ for primes $p$ and $r \in \N$ is unique 
(see for example, \cite{jantzen-schwermer}, Satz 5.14 and Satz 5.16),
hence, the Jordan decomposition of $D$ can only be built up from
odd blocks $\Z_q$ or even blocks $\Z_q \oplus \Z_q$ 
(no other prime and no other power occurs).
Let $e \geq 3$.
It does not matter how precisely the Jordan splitting of $D$ looks like, 
since $n \geq 3$, we can find a decomposition
$D = D_1 \orthplus D_2$ and there is at least one Jordan constituent in $D_1$ 
and there is at least one other Jordan constituent in $D_2$.
For $e \geq 3$, every Jordan constituent $C$ (no matter whether it is even or odd)
contains an isotropic vector of order $2$: Assume $C$ is even. 
Then there is a basis $C = \Z_q \gamma \oplus \Z_q \delta$.
If $C$ is of type (A), then $\gamma$ is isotropic. Hence, $2^{e-1}\gamma$ 
is isotropic as well and of order $2$.
If $C$ is of type (B), then still, $2^{e-1}\gamma$ is isotropic and of order $2$:
$$Q(2^{e-1}\gamma) = 2^{2(e-1)} Q(\gamma) = \frac{2^{2(e-1)}}{2^e} + \Z = 0 + \Z$$
as $2(e-1) \geq e$ as $e \geq 3 \geq 2$.
Suppose $C$ is an odd block. Then $C = \Z_q \gamma$ with $Q(\gamma) = \tfrac{a + v2^e}{2^{e+1}} + \Z$
and
$$ Q(2^{e-1} \gamma) = (a + v2^e) \frac{2^{2(e-1)}}{2^{e+1}} + \Z = (a + v2^e) \cdot (0 + \Z) = 0 + \Z$$
as $2(e-1) \geq e+1$ because $e \geq 3$. So, $2^{e-1}\gamma$ is isotropic of order $2$.
Summing it all up, we can take an isotropic vector of order $2$ from $D_1$ and another one from $D_2$.
As $D = D_1 \orthplus D_2$, those vectors are orthogonal and $\Z_2$--linearly independent.
Now let $e=1$ or $e=2$.

  Since the original rank was greater or equal to $7$, we can find a 'cut' through the Jordan splitting of $D$
	giving $D = D_1 \orthplus D_2$ and the rank (as $\Z_q$ module) of $D_1$ and $D_2$ both being $\geq 3$.
  Hence, we are done, if we show that in each of them, there is an isotropic vector of order $2^e$.
  
	So now let $D$ a $\Z_{2^e}$--module of rank greater or equal to $3$ 
	with a fixed Jordan splitting (see Thm. \ref{thm:jordan}).
  Let us denote the basis by $\mu_1, \delta_1, \mu_2, \delta_2, ..., \mu_r, \delta_r, \alpha_1, ..., \alpha_s$
  where the $\mu_i, \delta_i$ generate the even components and the $\alpha_i$ generate the odd components.
  Let the values of the quadratic form and bilinear form be $x_i, v_i, f_i$ as in Thm. \ref{thm:jordan},
	for example
  $Q(\alpha_i) = (f_i + v_i 2^e)/2^{e+1} + \Z$.
  Again we pass the problem to \Ztwo but this time we have to pass the "wrong" Gram matrix
  because of the division by two. More precisely we consider the matrix
  $$
  \tilde{G} := \mat{
  x_1 & 1   & \\
   1  & x_1 & \\
      &     & \ddots \\
      &     &         & x_r & 1 \\
      &     &         &  1  & x_r \\
      &     &         &     &     & a_1 + v_1 2^e \\
      &     &         &     &     &               & \ddots \\
      &     &         &     &     &               &        & a_r + v_r 2^e 
  }$$
  to be the Gram matrix of a bilinear form $\ideal{\cdot, \cdot}$ (without associated quadratic form!) of the
  abstract free \Ztwo module $\Ztwo^{2r + s}$.
  Notice that $\tilde{G}$ is invertible and
	hence, in the language of \cite{kneser}, Satz (15.8)
  this is a regular form. By this theorem, it splits a hyperbolic plane, i.e. there is a vector
  $\tilde{y} \in \Ztwo^{2r + s}$ such that $\ideal{\tilde{y},\tilde{y}} = 0$. Cancelling
  all $2$-powers in the coordinates of $\tilde{y}$ if necessary, we may assume $\tilde{y}$ to be primitive:
	if $\ideal{cv, cv} = 0$ for some $v \in \Ztwo^n$ and $c \in \Ztwo, c \neq 0$, then $0 = \ideal{cv, cv} = c^2 \ideal{v,v}$.
	As $\Ztwo$ is free of zero divisors, $\ideal{v,v} = 0$.
  Hence, there is a coordinate $i$ such that $\tilde{y}_i \in \Ztwo^\times$. Hence, if we take $R_{2^e}^\Z$ 
  coordinate wise and call the result $\overrightarrow{y} \in \Z^{2r + s}$ then there is an $i$ such that
  $\gcd(y_i, 2) = 1$. Thus, if we interpret $\overrightarrow{y}$ as an element $y \in D$ by writing 
	the coordinates in front of the participants of the Jordan basis, 
	$y$ is of order $2^e$ in $D$ 
	($x\cdot y_i \equiv 0 \mod 2^e \Rightarrow x \equiv 0 \mod 2^e$).
  Further, we claim that it is isotropic: We name the coordinates of $y$ and $\tilde{y}$ to be
  \begin{align*}
    \tilde{y}          &= (\tilde{a}_1, \tilde{b}_1, ..., \tilde{a}_r, \tilde{b}_r, \tilde{c}_1, ..., \tilde{c}_s) \\
    \overrightarrow{y} &= (a_1, b_1, ..., a_r, b_r, c_1, ..., c_s)
  \end{align*}
  then
  \begin{align*}
    Q(y) &= \sum_{i=1}^{r} a_i^2 Q(\mu_i) + b_i^2 Q(\delta_i) + a_i b_i (\mu_i, \delta_i) + \sum_{j=1}^s c_j^2 Q(\alpha_j) \\
         &= \sum_{i=1}^{r} \frac{x_i(a_i^2 + b_i^2) + 2a_i b_i}{2^{e+1}} + \sum_{j=1}^s \frac{c_j^2 (f_i + v2^e)}{2^{e+1}} + \Z
  \end{align*}
  Generally speaking, consider a term of the form
  $$\frac{xa^2}{2^{e+1}} + \Z$$ where $x \in \Z$ and $a = R^\Z_{2^e}(\tilde{a})$ for some $\tilde{a} \in \Ztwo$.
  If we wanted to write $R^\Z_{2^{e+1}}(\tilde{a})$ in place of $R^\Z_{2^e}(\tilde{a})$ 
  we would make some mistake of the form $v2^e$ but we have
  $$(a + v2^e)^2 \equiv a^2 + 2av2^e + v2^{2e} \equiv a^2 \mod 2^{e+1}$$
  hence,
  $$\frac{xR^\Z_{2^e}(\tilde{a})^2}{2^{e+1}} + \Z = \frac{xR^\Z_{2^{e+1}}(\tilde{a})^2}{2^{e+1}} + \Z$$
  so that after proceeding analogously with the terms $2 a_i b_i$ 
  (the missing $2$ is right in front and will get used twice!)
  the above sum becomes
  \begin{align*}
    Q(y) &= \sum_{i=1}^{r} 
               \frac{x_i(R^\Z_{2^{e+1}}(\tilde{a}_i)^2 + R^\Z_{2^{e+1}}(\tilde{b}_i)^2) + 2R^\Z_{2^{e+1}}(\tilde{a}_i) R^\Z_{2^{e+1}}(\tilde{b}_i)}
                    {2^{e+1}} \\
               &\quad + \sum_{j=1}^s \frac{R^\Z_{2^{e+1}}(\tilde{c}_i)^2 (f_i + v2^e)}{2^{e+1}} + \Z
  \end{align*}
  but as $R_{2^{e+1}}(\cdot) = R^\Z_{2^{e+1}}(\cdot) \mod 2^{e+1}$ is a ring homomorphism 
	(and we divide by no more than $2^{e+1}$ and have a '+\Z'), 
  this is nothing else than
  \begin{align*}
    Q(y) &= \frac{
              R^\Z_{2^{e+1}} \left(
                \sum_{i=1}^{r} 
                  x_i(\tilde{a}_i^2 + \tilde{b}_i^2) + 2\tilde{a}_i\tilde{b}_i
                + \sum_{j=1}^s \tilde{c}_i^2 (f_i + v2^e) \right)
            }{2^{e+1}} + \Z \\
         &= \frac{ R^\Z_{2^{e+1}} \ideal{\tilde{y}, \tilde{y}}}{2^{e+1}} + \Z \\
         &= \frac{ R^\Z_{2^{e+1}}(0)} {2^{e+1}} + \Z \\
         &= 0 + \Z
  \end{align*}
	All in all,
  $$y = a_1 \mu_1 + b_1 \delta_1 + ... + a_r \mu_r + b_r \delta_r + c_1 \alpha_1 + ... c_s \alpha_s$$
	is an isotropic element of order $2^e$.
\end{proof}

\section{Almost everything is an oldform}
\label{sec:main-theorem}
In this section we will show that if $D$ is 'too big', i.e. a part of the form $\Z_{p^e}$ is repeated too often, every modular form is an oldform.

\noindent Let 
\begin{align*}
  D= &\left(\Z_{2^1} \oplus \Z_{2^1} \right)^{e_1} \orthplus  \left(\Z_{2^2} \oplus \Z_{2^2} \right)^{e_2} \orthplus ...\\
     &  \orthplus (\Z_{2^1})^{o_1} \orthplus (\Z_{2^2})^{o_2} \orthplus ... \\
     &  \orthplus (\Z_{p_1^1})^{e_{1,1}} \orthplus (\Z_{p_1^2})^{e_{1,2}} \orthplus ... \orthplus (\Z_{p_1^{A_1}})^{e_{1,A_1}} \\
     &  \orthplus ...
\end{align*}
be a fixed Jordan splitting of a discriminant form $D$ (see Thm. \ref{thm:jordan}). 
The numbers $e_i$ are called the even $2^i$--ranks of $D$,
the numbers $o_i$ are called the odd $2^i$--rank of $D$ and
the numbers $e_{i,j}$ are called the $p_i^j$--ranks of $D$.

\noindent Jordan splittings of discriminant forms are not unique.
Let us assume that we have two different jordan splittings
\begin{align*}
  D= &\left(\Z_{2^1} \oplus \Z_{2^1} \right)^{e_1} \orthplus  \left(\Z_{2^2} \oplus \Z_{2^2} \right)^{e_2} \orthplus ...\\
     &  \orthplus (\Z_{2^1})^{o_1} \orthplus (\Z_{2^2})^{o_2} \orthplus ... \\
     &  \orthplus (\Z_{p_1^1})^{e_{1,1}} \orthplus (\Z_{p_1^2})^{e_{1,2}} \orthplus ... \orthplus (\Z_{p_1^{A_1}})^{e_{1,A_1}} \\
     &  \orthplus ...\\
	 = &\left(\Z_{2^1} \oplus \Z_{2^1} \right)^{e'_1} \orthplus  \left(\Z_{2^2} \oplus \Z_{2^2} \right)^{e'_2} \orthplus ...\\
     &  \orthplus (\Z_{2^1})^{o'_1} \orthplus (\Z_{2^2})^{o'_2} \orthplus ... \\
     &  \orthplus (\Z_{p_1^1})^{e'_{1,1}} \orthplus (\Z_{p_1^2})^{e'_{1,2}} \orthplus ... \orthplus (\Z_{p_1^{A_1}})^{e'_{1,A_1}} \\
     &  \orthplus ...
\end{align*}
By basic algebra, the decomposition of a finite abelian group into powers of 
$\Z_{p^e}$ for primes $p$ and $e \in \N$ is unique 
(see for example, \cite{jantzen-schwermer}, Satz 5.14 and Satz 5.16).
Consequently, we obtain
\begin{equation}
\label{eq:jordanDecompOddUnique}
e_{i,j} = e'_{i,j} ~ \text{for all $i,j \in \N \cup \{0\}$}
\end{equation}
and
\begin{equation}
\label{eq:jordanDecompEvenUnique}
o_j + 2e_j = o'_j + 2e'_j ~ \text{for all $j \in \N \cup \{0\}$}
\end{equation}
This implies the following:

\begin{rmk}
\label{rmk:uniquenessOfJordan}
Let $D$ be a discriminant form and $U, V$ $\Z$--submodules such that
$D = U \orthplus V$. It is easy to show that in this case, $U$ and $V$ are 
discriminant forms again (the crucial insight being that the restriction of 
$(\cdot, \cdot)$ to $U$ and $V$ is non-degenerate) and $V = U^\bot$.
Say $D = U \orthplus U^\bot$ possesses a Jordan decomposition $\Z_q \oplus ... \oplus \Z_q = (\Z_q)^c$ 
for one fixed prime power $q = p^{j_0}$. As $U, U^\bot$ are discriminant forms, 
using Thm. \ref{thm:jordan} we can choose Jordan 
decompositions of $U$ and $U^\bot$.
Let $e_j, o_j, e_{i,j}$ be the even and odd $2$-adic and $p$-adic ranks of 
$U$ and let $e_j', o'_j, e'_{i,j}$ be those of $U^\bot$.
Putting together the Jordan decompositions for $U$ and $U^\bot$ 
yields a new Jordan decomposition for $D$. Now we have two Jordan decompositions:
$$\Z_q \oplus ... \oplus \Z_q \cong D \cong U \orthplus U^\bot$$
so, by eqs \eqref{eq:jordanDecompOddUnique}, \eqref{eq:jordanDecompOddUnique},
if $p$ was odd, then $U, U^\bot$ also have Jordan decompositions 
$$U \cong (\Z_q)^a, U^\bot \cong (\Z_q)^b ~ \text{with $a+b=c$}$$
and if $p=2$ then all the Jordan constituents of 
$U$ and $U^\bot$ are only $2$-adic and either odd and of the form $\Z_{2^{j_0}}$ 
or even and of the form $\Z_{2^{j_0}} \oplus \Z_{2^{j_0}}$ with
  $$ o_{j_0} + 2e_{j_0} = o'_{j_0} + 2e'_{j_0}$$	
\end{rmk}

\noindent We summarize in the following Theorem:
\begin{thm}
  \label{thm:prank-geq-oldform}
  Let $\D = (D, Q)$ be a discriminant form with a fixed Jordan splitting as above.
  If there exists a prime $p = p_i$ and an exponent $e=e_{i,j}$ 
	(respectively exponent\textbf{s} $e=e_j, o=o_j$ if $p=2$) such that one of the following is true; 
  \begin{enumerate}[(i)]
    \item $p$ is odd and $j=1$ and $e \geq 7$
		\item $p$ is odd and $j>1$ and $e \geq 4$
		\item $p=2$ and $j=1$ or $j=2$ and $e+o \geq 9$
		\item $p=2$ and $j\geq 3$ and $e+o \geq 5$ 
  \end{enumerate}
	then every vector valued modular form for $D$ is an oldform.
\end{thm}
\begin{proof}
By Lemma \ref{lem:all-old}, it suffices to see that the algebraic part $\uparrow$ is 
surjective, so this is what we will show now.
We will prove that $\e_\gamma \in \im(\uparrow)$ for all $\gamma \in D$.
For doing this in turn, we are going to use Lemma \ref{lem:existence-nicely-orth-gamma-general}, so
it remains to show that 
\setlength{\jot}{0pt}
\begin{align}
	\begin{split}
    \label{eq:prank-geq-oldform:toshow}
	  &\text{For every $\gamma \in D$, there are is a prime $p$, an exponent $e$ and two}  \\
    &\text{isotropic, orthogonal, $\Z_{p^e}$--linearly independent vectors in $\gamma^\bot$.}
  \end{split}
\end{align}
Generally speaking, let $D = D_1 \orthplus D_2$ for two sub-discriminant forms $D_1$ and $D_2$.
For an element $\gamma \in D_1$, we can consider two orthogonal complements:
One of them is $\gamma^\bot = \{\delta \in D : (\gamma, \delta) = 0\}$, and the other one is
$\gamma^\bot \cap D_1$ the second one meaning that we 
ignore the fact that $\gamma$ comes from a bigger discriminant form and 
view $D_1$ as a discriminant form on its own. Suppose we can show that
\begin{align}
	\begin{split}
    \label{eq:prank-geq-oldform:reduction}
	  &\text{For every $\gamma_1 \in D_1$, there is a prime $p$, an exponent $e$ and two}  \\
    &\text{isotropic, orthogonal, $\Z_{p^e}$--linearly independent vectors inside $\gamma^\bot \cap D_1$.}
  \end{split}
\end{align} \setlength{\jot}{3pt}
then we deduce \eqref{eq:prank-geq-oldform:toshow}: Let $\gamma = \gamma_1 + \gamma_2$.
Observe that $\gamma_1^\bot \cap D_1 \subset \gamma^\bot$: 
For if $\delta_1 \in D_1$ satisfies $(\delta_1, \gamma_1) = 0 + \Z$ then
$$(\delta_1, \gamma) = (\delta_1, \gamma_1) + (\delta_1, \gamma_2) = 0 + 0 + \Z = 0 + \Z$$
Hence, the two vectors inside $\gamma_1^\bot \cap D_1$ are also in $\gamma^\bot$ and
the fact that they are isotropic and $\Z_{p^e}$--linearly independent does not depend 
on whether we view them as elements of $D_1$ or as elements of $D$.
Hence, all we need to do is verify \eqref{eq:prank-geq-oldform:reduction}, then the theorem is proved.
Let $D_\text{imp} \subset D$ be the 'important' part of the discriminant form as demanded by the theorem, 
i.e.
$$
D_\text{imp} \cong 
  \begin{cases} 
	  \Z_{p}^{7} & \text{if $p$ is odd and $j=1$} \\
		\Z_{p^j}^{4} & \text{if $p$ is odd and $j>1$} \\
		(\Z_{2^j} \times \Z_{2^j})^a \orthplus \Z_{2^j}^b & \text{if $p=2$ and $j=1$ or $j=2$, where $a,b$} \\
		                                            & \text{are arbitrary with the property that $a + b \geq 9$} \\
		(\Z_{2^j} \times \Z_{2^j})^a \orthplus \Z_{2^j}^b & \text{if $p=2$ and $j\geq 3$, where $a,b$} \\
		                                            & \text{are arbitrary with the property that $a + b \geq 5$} \\
  \end{cases}
$$
the exponentiation (i.e. the algebraic sum or 'times') being orthogonal.
Then $D = D_\text{imp} \orthplus D_\text{rest}$. As we only need to verify
\eqref{eq:prank-geq-oldform:reduction}, it suffices to show the existence of two 
isotropic, orthogonal, independent vectors in the complement (inside $D_\text{imp}$!) 
of every $\gamma \in D_\text{imp}$, so we will assume $D=D_{\text{imp}}$ from now on!

\noindent We put $q := p^j$ so that $D$ is a freely, finitely generated $\Z_q$--module.
We also let
\begin{equation}
\label{eq:prank-geq-oldform:n}
n := 
  \begin{cases} 
	  7 & \text{if $p$ is odd and $j=1$} \\
		4 & \text{if $p$ is odd and $j>1$} \\
		a+b \geq 9 & \text{if $p=2$ and $j=1,2$} \\
		a+b \geq 5 & \text{if $p=2$ and $j \geq 3$}
  \end{cases}
\end{equation}
be the rank of $D$.

\noindent We use Thm. \ref{thm:jordan} to get a fixed Jordan basis $\Gamma = \{\gamma_1, ..., \gamma_n\}$.
Rmk. \ref{rmk:change-of-basis-Zp-vs-ZmodpZ} implies that
\begin{equation}
\label{eq:prank-geq-oldform:Gram}
\text{\parbox[b]{\textwidth}{
The Gramian matrix $H = ((\gamma_i, \gamma_j))_{i,j=1,...,n}$ is $H = p^{-e} G + \Z$
for some symmetric matrix $G \in \Z^{n \times n}$ with its $p$-adic version 
$\tilde{G}$ being invertible, i.e. unimodular.
}}
\end{equation}
and that changes of bases over $\Zp$ induce changes of bases for $D$.
The bilinear form over \Zp induced by $\tilde{G}$ will 
be denoted by $\ideal{\cdot, \cdot}_p$.
We make it clear once and for all that this does not immediately correspond to $(\cdot, \cdot)$,
for example: if $(\gamma, \delta) = a/q + \Z$ for some $a \in \Z$,
then all that we know is that there exists an $m \in \Z$ such that 
$\pform{\tilde{\gamma}, \tilde{\delta}} = \tilde{a} + mq$.

\noindent Let $0 \neq \gamma \in D$ be arbitrary (the case $\gamma=0$ is handled afterwards).
The proof consists of two steps: We compute 
a decomposition $D = U \orthplus U^\bot$ such that $\gamma \in U$.
Then $U^\bot \subset \gamma^\bot$ and we will find two vectors as announced inside $U^\bot$.
\\
\\
\noindent \underline{Case 1: $(\gamma, \gamma) = a/q + \Z$ with $(a,p)=1$.}

\noindent Let $\gamma = \smallsum_i a_i \gamma_i$ with a fixed choice $a_i \in \Z$.
We put $\tilde{\gamma} = (\iota(a_1), ..., \iota(a_n))$, where $\iota$ is the imbedding 
$\Z \hookrightarrow \Zp$.
Then $\ideal{\tilde{\gamma}, \tilde{\gamma}}_p = a + mq$ for some $m \in \Z$.
As $(a, p)=1$ and $p|q$, $a+mq$ is still a unit in $\Zp$. Hence,
the Gram matrix of the submodule $\Zp \tilde{\gamma}$ is just
the $1$--by--$1$--matrix $(a + mq) \in \GL_1(\Zp)$.
By \cite{kneser}, Satz 1.6 on p.2, we can split $\tilde{\gamma}$ off orthogonally.
Hence, using Rmk. \ref{rmk:change-of-basis-Zp-vs-ZmodpZ}, we can split 
off $\gamma$ orthogonally from $D$, i.e. for $U := \Z_q \gamma$ we have
$D = U \orthplus U^\bot$ with $U^\bot = \gamma^\bot$.
By Rmk. \ref{rmk:uniquenessOfJordan}, $U^\bot \cong (\Z_q)^{n-1}$
where, by \eqref{eq:prank-geq-oldform:n},
$$n-1 \geq 
  \begin{cases} 
	  6 \geq 5 & \text{if $p$ is odd and $j=1$} \\
		3 \geq 2 & \text{if $p$ is odd and $j\geq 2$} \\
	  8 \geq 7 & \text{if $p=2$ and $j=1$ or $2$} \\
		4 \geq 3 & \text{if $p=2$ and $j\geq 3$}
	\end{cases}
$$
Thus we may apply Lemmas \ref{lem:finish:p=odd} in the odd case and \ref{lem:finish:p=2} in the case $p=2$
to get two isotropic, orthogonal, linearly independent vectors inside $U^\bot$. We are done in this case.
\\
\\
\noindent \underline{Case 2: $(\gamma, \gamma) = a/q + \Z$ with $p|a$.}

\noindent If $\gamma \neq 0$ is not primitive, then we can write $\gamma = p^s \mu$ for some $s \in \N$ 
and a primitive $\mu$. We have $\mu^\bot \subset \gamma^\bot$ as 
for every $\delta \in \mu^\bot$,
$(\delta, \gamma) = p^s (\delta, \mu) = p^s \cdot (0+\Z) = 0+\Z$. 
If $(\delta, \delta) = b/q + \Z$ with $b \in \Z$ such that 
$p\ndivides b$, then by the first case, we find two isotropic, 
orthogonal, linearly independent vectors inside 
$\mu^\bot \subset \gamma^\bot$. Hence, we are done in this case. 
Now assume that $p|b$.
Let $\mu = \smallsum_i a_i \gamma_i$ with a fixed choice $a_i \in \Z$.
We put $\tilde{\mu} = (\iota(a_1), ..., \iota(a_n)) \in \Zp^n$
where $\iota$ denotes the formal imbedding $\Z \hookrightarrow \Zp$.
The fact that $\mu$ is primitive in $D$ translates into the 
condition that $\tilde{\mu}$ is primitive over $\Zp$.
We have $\pform{\tilde{\mu}, \tilde{\mu}} = b + mq$
for some $m \in \Z$. Now
$p|b$ and $p|q$ so $\nu_p(\pform{\tilde{\mu}, \tilde{\mu}}) > 0$,
where $\nu_p$ denotes the $p$-adic valuation.
Also, $\tilde{G}$, the gramian matrix of $\pform{\cdot, \cdot}$, 
is invertible over $\Zp$ by \eqref{eq:prank-geq-oldform:Gram}.
Hence, we can apply Lemma \ref{lem:primitive-high-norm-splits} to
split off $\Zp \tilde{\mu} \oplus \Zp \tilde{\delta}$ 
for some $\tilde{\delta} \in \Zp^n$ orthogonally.
After going back to $\Z_q$ using Rmk. \ref{rmk:change-of-basis-Zp-vs-ZmodpZ}, 
$D$ splits orthogonally into
$$D = U \orthplus U^\bot$$
where $U = \Z_q \gamma' \oplus \Z_q \delta'$ for some $\delta' \in D$.
By Rmk. \ref{rmk:uniquenessOfJordan}, $U^\bot \cong (\Z_q)^{n-2}$ where,
by \eqref{eq:prank-geq-oldform:n},
 $$n-2 \geq 
  \begin{cases} 
	  5 & \text{if $p$ is odd and $j=1$}\\
	  2 & \text{if $p$ is odd and $j \geq 2$}\\
	  7 & \text{if $p=2$ and $j=1$ or $2$} \\
		3 & \text{if $p=2$ and $j\geq 3$}
	\end{cases}
$$
Thus we may apply Lemmas \ref{lem:finish:p=odd} in the odd case and \ref{lem:finish:p=2} in the case $p=2$
to get two isotropic, orthogonal, linearly independent vectors inside 
$U^\bot \subset (\gamma')^\bot \subset \gamma^\bot$. 
We are done in this case.

It remains to see what happens if $\gamma=0$.
We take any other $\delta \neq 0$ and proceed as above
to find two orthogonal, isotropic, linearly independent 
vectors inside $\delta^\bot$. Then, they are also contained
in $\gamma^\bot$ as $\gamma^\bot$ is all of $D$!

\end{proof}

\begin{cor}
\label{cor:main-cor}
If $N\in \N$ is fixed and $\D$ is a discriminant form of level $N$ with 
$|D| \geq N^9$, then every vector valued modular form for $D$ is an oldform. 
This bound ($N^9$) is absolutely not optimal.
\end{cor}
\begin{proof}
By measuring the size of a Jordan decomposition, we see that at least one 
$\Z_{p^e}$--part has to occur with a multiplicity $\geq 9$.
Thus, the assumption of Thm. \ref{thm:prank-geq-oldform} is met.
\end{proof}
\begin{appendix}
\section{Appendix}

\begin{rmk}
  \label{rmk:change-of-basis-Zp-vs-ZmodpZ}
  Let $p$ be a prime (not necessarily odd) and $D$ a discriminant form such that 
	$D \cong \Z_{p^e}^n$ (as groups, ignoring the quadratic form). 
  Choose a basis $\mathcal{A} = \{\alpha_1, ..., \alpha_n\}$. 
	Put 
	  $$H := ((\alpha_i, \alpha_j))_{i,j=1,...,n} \in (\Q/\Z)^{n \times n}$$
	then $H$ is of the form
  $H = p^{-e} G + \Z$
  for some symmetric matrix $G \in \Z^{n \times n}$. 
	Although, $G$ may not be invertible over $\Z$, 
	its $p$-adic version $\tilde{G}=\iota(G) \in \Zp^{n \times n}$ is.
	Here, $\iota$ denotes the formal imbedding $\Z \hookrightarrow \Zp$.
	Assume we are given a change of basis over $\Zp$, essentially a matrix 
  $\tilde{S} \in \GL_n(\Zp)$ then $S := R_{p^e}^\Z(\tilde{S})$ (see Not. \ref{not:padicStuff}) is 
  a change of basis over $\Z_{p^e}$ in the following sense:
  If
  $$S = \mat{
          s_{11} & s_{12} & \hdots & s_{1n} \\
          s_{21} & s_{22} & \hdots & s_{2n} \\
          \vdots & \vdots & \vdots & \vdots \\ 
          s_{n1} & s_{n2} & \hdots & s_{nn}
        }$$ 
  then we may put
  \begin{align*}
    \beta_1 &:= R_{p^e}^\Z(s_{11}) \alpha_1 + R_{p^e}^\Z(s_{21}) \alpha_2 + ... + R_{p^e}^\Z(s_{n1}) \alpha_n \\
    \beta_2 &:= R_{p^e}^\Z(s_{12}) \alpha_1 + R_{p^e}^\Z(s_{22}) \alpha_2 + ... + R_{p^e}^\Z(s_{n2}) \alpha_n \\
             & \vdots \\
    \beta_n &:= R_{p^e}^\Z(s_{1n}) \alpha_1 + R_{p^e}^\Z(s_{2n}) \alpha_2 + ... + R_{p^e}^\Z(s_{nn}) \alpha_n \\
  \end{align*}
  i.e. the coordinates of the new basis (w.r.t. the old basis) are the columns
  of $R_{p^e}^\Z(S)$. Then $\mathcal{B} = \{\beta_1, ..., \beta_n\}$ is a new basis for $D$ 
	and the Gram matrix of them is
  $$p^{-e} R_{p^e}^\Z(\tilde{S}^T \tilde{G} \tilde{S}) + \Z$$
\end{rmk}
\begin{proof}

We use Thm. \ref{thm:jordan} to get a Jordan decomposition of $D$.
By basic algebra, the decomposition of a finite abelian group into powers of 
$\Z_{p^e}$ for primes $p$ and $e \in \N$ is unique 
(see for example, \cite{jantzen-schwermer}, Satz 5.14 and Satz 5.16).
Hence, the Jordan decomposition only consists (algebraically) of
direct summands of the form $\Z_{p^e}$. Hence, by Thm. \ref{thm:jordan}, if $p$ is odd then
$H = p^{-e} \diag(a_1, ..., a_n)$ with $a_i \in \Z, (a_i,p)=1$.
If $p=2$ then $H$ is of the form
$$
  H := 2^{-e}\mat{
  x_1 & 1   & \\
   1  & x_1 & \\
      &     & \ddots \\
      &     &         & x_r & 1 \\
      &     &         &  1  & x_r \\
      &     &         &     &     & a_1 \\
      &     &         &     &     &               & \ddots \\
      &     &         &     &     &               &        & a_r  
  } + \Z$$
with $x_i \in \{0,2\}$ and $(a_i,2) = 1$.
Now for
$G = \diag(a_1, ..., a_n)$ if $p$ is odd, respectively,
$$
  G := \mat{
  x_1 & 1   & \\
   1  & x_1 & \\
      &     & \ddots \\
      &     &         & x_r & 1 \\
      &     &         &  1  & x_r \\
      &     &         &     &     & a_1 \\
      &     &         &     &     &               & \ddots \\
      &     &         &     &     &               &        & a_r  
  }$$
if $p=2$, we have $G\in \Z^{n \times n}$ is symmetric, $(\det(G),p)=1$, so 
$\tilde{G} = \iota(G) \in \Zp^{n \times n}$ is invertible and
$H = p^{-e} G + \Z$.
Remark that one can also show this directly (without the usage of a Jordan basis)
but it involves some fumbling with different maps and relations between
$\Z, \Zp$ and $\Z_{p^e}$.
	Now we show the assertion about the change of basis:
  As $\tilde{S}$ is invertible over \Zp, there exists $\tilde{S}^{-1} \in \Zp^{n \times n}$.
	When we have a commutative ring $R$, a matrix $X \in R^{n \times n}$ and an ordered set of vectors
  $v = \{v_1, ..., v_n\} \subset R^n$ then we say that we operate on $v$ if we form
  $w_i := \sum_{j=1}^n X_{ji} v_i$ i.e. the new coordinates are given column wise.
  We write $w = X.v$ in this case.
  A quick matrix multiplication reveals that $Y.X.v = (XY).v$ for all matrices 
	$X,Y \in R^{n \times n}$ and every ordered set of vectors $v$.
  Hence, in our situation above, operating on the new basis $\mathcal{B} = \{\beta_1, ..., \beta_n\} = R_{p^e}(\tilde{S}).\mathcal{A}$ 
  with $R_{p^e}(\tilde{S}^{-1})$ results in
  $$R_{p^e}(\tilde{S}^{-1}).R_{p^e}(\tilde{S}).\mathcal{A} = \left( R_{p^e}(\tilde{S}^{-1})R_{p^e}(\tilde{S}) \right).\mathcal{A}
       = R_{p^e}(\tilde{S}^{-1}\tilde{S}).\mathcal{A} = \id.\mathcal{A} = \mathcal{A}$$
  Hence, the $\alpha_i$ lie in the $\Z$-span of the $\beta_i$, so the $\beta_i$ generate the full module $D$.
  We need to see that they are $\Z_{p^e}$--linearly independent.
  Assume there is a relation $\sum_i \lambda_i \beta_i = 0$, then
  \begin{align*}
    0 &= \sum_i \lambda_i \beta_i = \sum_i \lambda_i \sum_j s_{ji} \alpha_j \\
      &= \sum_j \underbrace{\left(\sum_i s_{ji} \lambda_i\right)}_{= (S\cdot\lambda)_i} \alpha_i \\
  \end{align*}
  As the $\alpha_i$ formed a basis, $S \cdot \lambda$ is the zero vector over $\Z_{p^e}$.
  Now we know that $\det(\tilde{S}) \in \Zp^\times$.
  As the reduction maps are ring homomorphisms and $\det$ is a polynomial,
  $\det(S) = R_{p^e}(\det(\tilde{S})) \in R_{p^e}(\Zp^\times) \subset \Z_{p^e}^\times$. 
	Hence, $S$ is invertible and $S\lambda=0$ implies $\lambda=0$.
  We also compute
  \begin{align*}
    (\beta_i, \beta_j) &= (\sum_x s_{xi} \alpha_i, \sum_y s_{yj} \alpha_j) \\
                       &= \sum_x \sum_y s_{xi} s_{yj} H_{xy} \\
                       &= \sum_x \sum_y s_{xi} s_{yj} p^{-e}G_{xy} + \Z \\
                       &= \frac{(S^T G S)_{xy}}{p^e} + \Z
  \end{align*}
  By definition, we have $R_{p^e}(G) \equiv G \mod p^e$ and hence,
  $$S^T G S \equiv R_{p^e}(\tilde{S}^T) R_{p^e}(\tilde{G}) R_{p^e}(\tilde{S}) \equiv R_{p^e} (\tilde{S}^T \tilde{G} \tilde{S}) \mod p^e$$
  so that
  $\frac{(S^T G S)_{xy}}{p^e} + \Z = \frac{R_{p^e}(\tilde{S}^T \tilde{G} \tilde{S})}{p^e} + \Z$.
\end{proof}

\begin{thm}
  \label{thm:only-2-unimodular-forms-over-odd-p}
  Let $p$ be an odd prime and let $R = \Zp$ or $R = \Z_{p^e}$ for some $e \in \N$.
  Up to isomorphism, there are only two non-degenerate, unimodular, 
  symmetric bilinear forms over $R$. Further, for any two forms 
  $B, B'$ on a freely, finitely generated $R$-module $V$, one has
  $$ B \cong B' \iff \det(B) \equiv \det(B') \mod (R^\times)^2$$
\end{thm}
\begin{proof}
  If we can transform a form isomorphically into another form over $\Zp$ then by
  Remark \ref{rmk:change-of-basis-Zp-vs-ZmodpZ}, we can do so over $\Z_{p^e}$ (this is even true for $p=2$)
  so it suffices to show the assertion for $\Zp$.
  The reason why this fails for $p=2$ is that the one-dimensional situation, i.e. $\Ztwo^\times / (\Ztwo^\times)^2$ 
  is more complicated than for $p$ odd.
  If $A,B \in \Zp^{n \times n}$ then we write $A \sim B$ iff. there exists an $S \in \GL_n(\Zp)$ 
  such that $S^T A S = B$. This is an equivalence relation encapturing isomorphy of bilinear forms, i.e.
  two bilinear forms $(\cdot,\cdot)_1, (\cdot,\cdot)_2$ are isometrically isomorphic iff. 
  their Gram matrices are in $\sim$ relation.
  If $p$ is odd and $G$ is the Gram matrix of a unimodular symmetric bilinear form, then
  we can apply the machinery in \cite{conway-sloane} Chapter 15, \textsection 4.4, pp. 396-397
  to see that there is an $S \in \GL_n(\Zp)$ such that $D = S^T G S$ is diagonal.
  Comparing the determinants, we see that all elements on the diagonal of $D$ are units in \Zp.
  We know that there is a fixed number $t \in \Z$ such that $t$ is not a square in $\Z_p$ and
  \begin{equation}
    \label{eq:sq-classes-in-Zp}
    \Zp^\times / (\Zp^\times)^2 = \{1(\Zp^\times)^2, t(\Zp^\times)^2\}
  \end{equation}
  see for example
  \cite{cassels}, Cor. on p.40 or any book on $p$-adic numbers.
  This means that for every diagonal entry $a$ in $D$, there exists a unit 
  $\epsilon \in \Zp^\times$ such that $\epsilon^2 a = 1$ or $\epsilon^2 a = t$.
  Writing those $\epsilon$ diagonally in some matrix $S_2$ and resorting all the $1$'s 
  to the top left, we see that
  $G \sim \diag(1,...,1,t,...,t)$ with, say $a$ ones and $b$ $t$'s.
  Now we show that
  \begin{equation}
    \label{eq:two-dim-sim-overo-p-adics}
    \mat{1 & 0 \\ 0 & 1} \sim \mat{t & 0 \\ 0 & t}
  \end{equation}
  $\Z_p^\times$ is cylic (basic algebra!). Let $x$ be a generator. Then
  $x$ is not a square:
  If $x$ was a square then in fact, every unit would be a square but this is impossible
  because the group homomorphism $s:\Z_p^\times \to \Z_p^\times, a \mapsto a^2$ is not injective,
  in fact, for the generator $x$ as above, we have $y := x^{(p-1)/2} \neq 1$ but $y \in \ker(s)$. 
  Hence, it can also not be surjective by the pigenhole principle.
  Let $\mathcal{S}$ be the set of squares and let $\mathcal{N}$ 
	be the set of nonsquares in $\Z_p^\times$. As $x$ is not a square,
  $\mathcal{S} = \{x^{2e} : e=0,1,...,(p-3)/2\}$ and $\mathcal{N} = \{x^{2e + 1} : e=0,1,...,[(p-3)/2]+1\}$.
  In particular, $|\mathcal{S}| = |\mathcal{N}| = \tfrac{p-1}{2}$. 
  Consider $h : \Z_p \to \Z_p, h(a) = t - a^2$.
  Assume for a moment that $h(a) \notin \mathcal{S}$ for all $a \in \Z_p$. 
  Then $h(a) \in \mathcal{N} \cup \{0\}$ for all $a \in \Z_p^\times$, 
  but $h(a) = 0$ implies $t = a^2$ which is impossible as $t$ was a nonsquare.
  Hence, $h(a) \in \mathcal{N}$ for all $a \in \Z_p$.
  Define an equivalence relation on $\Z_p$ by $a \sim b \iff a = \pm b$.
  Then $\Z_p/\sim$ consists of the classes $\{0\}$ and $\{a, -a\}$ for $a \in \Z_p^\times$,
  i.e. $|\Z_p/\sim| = \tfrac{p-1}{2} + 1$.
  $h$ becomes a well defined map $\overline{h}$ on $\Z_p/\sim$ and now $\overline{h}$ is injective!
  We conclude that $\tfrac{p-1}{2} + 1 = |\Z_p/\sim| = |\im(\overline{h})| = |\im(h)| \leq |\mathcal{N}| = \tfrac{p-1}{2}$, 
  a contradiction. Thus, there exists a $b \in \Z_p$ and $a \in \Z_p^\times$ such that $t - b^2 = h(b) = a^2$.
  This is equivalent to saying that $a^2 + b^2 = t$.
  Take arbitrary but fixed representatives in $\Z$ of $a,b$ (also called $a,b$ in the sequel).
  As the reduced $a$ was a unit in $\Z_p$, $a \in \Zp^\times$.
  As $R_p(a^2 + b^2) \equiv t \,\,\cancel{\equiv}\,\, 0 \mod p$, $a^2 + b^2 \in \Zp^\times$ 
  (the $0$-th term in the $p$-adic expansion of $a^2 + b^2$ is precisely $R_p(a^2 + b^2)$!).
  By \eqref{eq:sq-classes-in-Zp}, there are only two possibilities, either the square class of 
  $a^2 + b^2$ is $1(\Zp^\times)^2$ or $t(\Zp^\times)^2$. Actually, 
  the latter one must be the case as
  $\epsilon^2 (a^2 + b^2) = 1 \Rightarrow 
     1 \equiv R_p(1) \equiv R_p(\epsilon)^2 R_p(a^2 + b^2) \equiv R_p(\epsilon)^2 t$
  so $t \equiv (\epsilon^{-1})^2 \mod p$ is a square. Contradiction.
  Hence, there exists an $\epsilon \in \Zp^\times$ such that
  $\epsilon^2(a^2 + b^2) = t$ or phrased differently, there are $A,C \in \Zp$ such that
  $$A^2 + C^2 = t ~ \text{and}~ A \in \Zp^\times$$
  ($A \in \Zp^\times$ as this $A$ is $\epsilon\cdot a$ and the reduced version of $a$ was in $\Z_p^\times$)
  Consider
  $$S = \mat{A & -\tfrac{AC}{t} \\ C & -\tfrac{C^2}{t} + 1} = \mat{A & 0 \\ C & 1} \mat{1 & -Ct^{-1} \\ 0 & 1}$$
  By the decomposition, $S$ is invertible, so
  $$G \sim S^T G S = \mat{t & 0 \\ 0 & -C^2t^{-1} + 1}$$
  Comparing the square classes of the determinants, we see that 
  $1(\Zp^\times)^2 \equiv \det(G) \equiv t\cdot(-C^2t^{-1} + 1) \mod (\Zp^\times)^2$
  so there exists an $\epsilon \in \Zp^\times$ such that $\epsilon^2 (-C^2t^{-1} + 1) = t$ and thus
  $$S' := S \cdot \mat{1 & 0 \\ 0 & \epsilon}$$
  is such that $(S')^T G S' = \tmat{t & 0 \\ 0 & t}$. 
  Equation \eqref{eq:two-dim-sim-overo-p-adics} is shown.
  We stopped at the point where $G \sim \diag(1,...,1,t,...,t)$. By equation \eqref{eq:two-dim-sim-overo-p-adics}
  we can turn every pair of $t$'s into a pair of $1$'s thus arriving at
  $G \sim \diag(1,...,1)$ or $G \sim \diag(1,...,1,t)$ depending on whether the amount of $t$'s was even or odd.
  This form is called the canonical form of $G$.
  Comparing the square classes of the determinants we see that
  $$G \sim \begin{cases}
             \diag(1,...,1)   & \text{if $\det(G) \equiv 1 \mod (\Zp^\times)^2$} \\
             \diag(1,...,1,t) & \text{if $\det(G) \equiv t \mod (\Zp^\times)^2$}
           \end{cases}
  $$
  Now the assertion is proved: two unimodular forms with coinciding square 
  classes of unit determinants have the \textbf{same} canonical form. 
  In particular, they are isomorphic.
\end{proof}

\end{appendix}

\end{document}